\documentclass[reqno]{amsart}

\usepackage[utf8]{inputenc}
\usepackage{tikz,amsthm,amsmath,amssymb,mathtools}
\usepackage{amsmath}
\usepackage{amsfonts}
\usepackage{amssymb}
\usepackage{pdfsync}
\usepackage{color,graphicx,tikz-cd}
\usepackage{a4wide}
\usepackage{amscd,amsthm,latexsym}
\usepackage{tikz,mathtools}
\usepackage{hyperref}
\usepackage{relsize}
\usepackage{tikz}
\usetikzlibrary{math}
\usepackage{float}
\usepackage{subcaption}
\usepackage[misc]{ifsym}
\usepackage{caption}
\usepackage{enumitem}
\usepackage{bm}
\usepackage{graphicx}
\usepackage{ifthen}
\usetikzlibrary{snakes}
\usepackage{setspace}

\title{On linearization coefficients of $q$-Laguerre polynomials}

\author{Byung-Hak Hwang}
\address{Department of Mathematics, Seoul National University, Seoul,
South Korea}
\email{xoda@snu.ac.kr}

\author{Jang Soo Kim}
\address{
Department of Mathematics, Sungkyunkwan University, Suwon,
South Korea}
\email{jangsookim@skku.edu}

\author{Jaeseong Oh}
\address{Department of Mathematics, Seoul National University, Seoul,
South Korea}
\email{jaeseong\_oh@snu.ac.kr}

\author{Sang-Hoon Yu}
\address{Department of Mathematics, Seoul National University, Seoul,
South Korea}
\email{ysh4017@snu.ac.kr}

\date{\today}

\date{\today}

\newtheorem{thm}{Theorem}[section]
\newtheorem{lem}[thm]{Lemma}
\newtheorem{prop}[thm]{Proposition}

\theoremstyle{definition}
\newtheorem{exam}[thm]{Example}
\newtheorem{defn}[thm]{Definition}

\newtheorem{remark}[thm]{Remark}

\newcommand\CR{\operatorname{cross}}

\newcommand\LL{{\mathcal{L}}}

\newcommand\cross{{\operatorname{cr}}}
\newcommand\ov{{\operatorname{ov}}}
\newcommand\wex{{\operatorname{wex}}}
\newcommand\bind{\operatorname{bindex}}
\newcommand\bdiff{\operatorname{bdiff}}
\newcommand\bwex{{\operatorname{bwex}}}

\newcommand\ee{\operatorname{e}}

\newcommand\M{{\operatorname{M}}}
\newcommand\PM{{\operatorname{PM}}}

\newcommand\bwt{\operatorname{bwt}}
\newcommand\wt{\operatorname{wt}}

\newcommand{\mm}{\mathfrak{m}}
\newcommand{\DD}{\mathcal{D}}
\newcommand{\ID}{\operatorname{Id}}

\tikzset{
    marked/.style={densely dashed, draw=black},
    marked0/.style={densely dashed, draw=gray!80},
    unmarked/.style={ draw=black},
    block/.style={fill=gray!15},
}

\begin{document}

\begin{abstract}
The linearization coefficient $\mathcal{L}(L_{n_1}(x)\dots L_{n_k}(x))$ of classical Laguerre polynomials $L_n(x)$ is known to be equal to the number of $(n_1,\dots,n_k)$-derangements, which are permutations with a certain condition. Kasraoui, Stanton and Zeng found a $q$-analog of this result using $q$-Laguerre polynomials with two parameters $q$ and $y$. Their formula expresses the linearization coefficient of $q$-Laguerre polynomials as the generating function for $(n_1,\dots,n_k)$-derangements with two statistics counting weak excedances and crossings. In this paper their result is proved by constructing a sign-reversing involution on marked perfect matchings.
\end{abstract}

\maketitle

\section{Introduction}

A family of polynomials $P_n(x)$ are called \emph{orthogonal polynomials} with
respect to a linear functional $\LL$ if $\deg P_n(x)=n$ for $n\ge0$ and
$\LL(P_m(x)P_n(x)) = 0$ if and only if $m\ne n$. The $n$th \emph{moment}
$\mu_n$ of the orthogonal polynomials is defined by $\mu_n = \LL(x^n)$. It is
well known that monic orthogonal polynomials $P_n(x)$ satisfy a three-term
recurrence of the form
\begin{equation}
  \label{eq:3}
P_{n+1}(x) = (x-b_n) P_n(x) - \lambda_n P_{n-1}(x).  
\end{equation}

Viennot \cite{ViennotLN} developed a combinatorial theory to study orthogonal
polynomials. In particular, he showed that orthogonal polynomials $P_n(x)$ and
the moments $\mu_n$ are expressed as weighted sums of certain lattice paths.
There are several classical orthogonal polynomials whose moments have simple
combinatorial meanings. For example, the $n$th moment of the Hermite
(respectively, Charlier and Laguerre) polynomials is the number of perfect
matchings (respectively, set partitions and permutations) on
$[n]:=\{1,2,\dots,n\}$.

By definition of orthogonal polynomials, it is easily seen that 
\[
  P_m(x)P_n(x) = \sum_{\ell=0}^{m+n} c^{\ell}_{m,n} P_\ell(x), \qquad
 c^{\ell}_{m,n} = \LL(P_\ell(x)P_m(x)P_n(x))/\LL(P_\ell(x)^2).
\]
Thus the coefficients $c^{\ell}_{m,n}$ can be computed using the quantities
$\LL(P_{n_1}(x)\dots P_{n_k}(x))$. We call $\LL(P_{n_1}(x)\dots P_{n_k}(x))$ a
\emph{linearization coefficient}.

For the above mentioned classical orthogonal polynomials, the linearization
coefficients also have nice combinatorial interpretations as follows. Let $n_1, n_2, \dots n_k$ be positive integers with $n=n_1+\cdots+n_k$, and consider the set $I_i$ ($i=1,2,\dots,k$) of consecutive integers from $n_1+\cdots+n_{i-1}+1$ to $n_1+\cdots+n_i$, where $n_0=0$. If $P_n(x)$ are the
Hermite (respectively, Charlier and Laguerre) polynomials, then
$\LL(P_{n_1}(x)\dots P_{n_k}(x))$ is the number of inhomogeneous perfect matchings (respectively, set partitions and permutations) on
$I_1\sqcup\dots\sqcup I_k=[n]$, see \cite{CKS,Even1976,FZ84, Foata1988, Zeng1992}
and references therein. Here, a perfect matching $\mm$ (respectively, set partition
$\pi$ and permutation $\sigma$) is \emph{inhomogeneous} if there are no edges
(respectively, two elements in the same block and two elements $j$ and
$\sigma(j)$) that are contained in the same set $I_i$.

There are $q$-analogs of the above combinatorial formulas for linearization
coefficients of Hermite, Charlier and Laguerre polynomials due to Ismail,
Stanton and Viennot \cite{ISV}, Anshelevich \cite{Anshelevich} and Kasraoui,
Stanton and Zeng \cite{Kasraoui2011}, respectively. There is a unified way to
prove combinatorial formulas for linearization coefficients using so called ``separation of variables" \cite{IKZ13}. We refer the reader to the survey \cite{CKS} for more details on these linearization coefficients.

Suppose that $P_n(x)$ are orthogonal polynomials whose moments $\LL(x^n)$ have a
combinatorial model as in the case of Hermite, Charlier or Laguerre polynomials.
Since $P_n(x)$ satisfy a simple recurrence \eqref{eq:3}, one may also give a
combinatorial model for $P_n(x)$ with possibly negative signs involved. These
combinatorial models for $P_n(x)$ and $\LL(x^n)$ naturally yield a combinatorial
meaning to $\LL(P_{n_1}(x)\dots P_{n_k}(x))$, which may have negative signs.
Therefore, if there is a combinatorial formula for $\LL(P_{n_1}(x)\dots
P_{n_k}(x))$ with only positive terms, the most satisfying combinatorial proof
of this formula would be finding a sign-reversing involution on the
combinatorial models for $\LL(P_{n_1}(x)\dots P_{n_k}(x))$ whose fixed points
give the positive terms in the formula.

Indeed, the formulas for linearization coefficients of $q$-Hermite \cite{ISV}
and $q$-Charlier polynomials \cite{Anshelevich} have been proved in this way by
Ismail, Stanton and Viennot \cite{ISV} and Kim, Stanton and Zeng \cite{KSZ}.
However, such a proof is missing in the case of $q$-Laguerre polynomials. In
this paper, we prove the formula for linearization coefficients of
$q$-Laguerre polynomials due to Kasraoui, Stanton and Zeng \cite{Kasraoui2011} by
finding a sign-reversing involution. We now describe their result below. 

The \emph{$q$-Laguerre polynomials} $L_{n}(x;q,y)$ are defined by the
three-term recurrence relation
\begin{equation} \label{eq:general_recur}
L_{n+1}(x;q,y) = (x-y[n+1]_q-[n]_q) L_{n}(x;q,y) - y[n]^2_q L_{n-1}(x;q,y)
\end{equation}
with $L_0(x;q,y) = 1$ and $L_1(x;q,y) = x-y$. Here, we use the notation $[n]_q =
1+q+\dots+q^{n-1}$. From now on $\LL$ denotes the linear functional with respect to which
the $q$-Laguerre polynomials are orthogonal.

The set of permutations of $[n]$ is denoted by $S_n$. For $\sigma\in S_n$, a
\emph{weak excedance} of $\sigma$ is an integer $i\in[n]$ such that
$\sigma(i)\ge i$. A \emph{crossing} of $\sigma$ is a pair $(i,j)$ of integers
$i,j\in[n]$ such that $i<j\le\sigma(i)<\sigma(j)$ or $\sigma(i)<\sigma(j)<i<j$.
We denote by $\wex(\sigma)$ (respectively, $\CR(\sigma)$) the number of weak
excedances (respectively, crossings) of $\sigma$. For positive integers
$n_1,\dots,n_k$ and $N=n_1+\dots+n_k$, an
\emph{$(n_1,\dots,n_k)$-derangement} is a permutation $\sigma\in
S_{N}$ such that there is no integer $i\in[N]$ with
\[
n_1+\dots+n_{j-1}+1 \le i,\sigma(i)\le n_1+\dots+n_j
\]
for some $j\in [k]$.
The set of $(n_1,\dots,n_k)$-derangements is denoted by $\DD(n_1,\dots,n_k)$.

Kasraoui, Stanton and Zeng
\cite{Kasraoui2011} showed that the $n$th moment is given by
\begin{equation}
  \label{eq:mu1}
\mu_n(q,y) = \LL(x^n) = \sum_{\sigma\in S_n} y^{\wex(\sigma)} q^{\CR(\sigma)}.
\end{equation}
They also proved the following formula for the linearization coefficients of
$q$-Laguerre polynomials.

\begin{thm}$($\cite{Kasraoui2011}$)$
\label{thm:Main}
The linearization coefficients of $q$-Laguerre polynomials are given
by 
$$\mathcal{L}(L_{n_1}(x;q,y)\cdots L_{n_k}(x;q,y))= \sum_{\sigma\in \mathcal{D}
  (n_1,\dots,n_k)} y^{\wex(\sigma)}q^{\CR(\sigma)}.$$
\end{thm}

In \cite{Kasraoui2011} using recurrence relation for $\mathcal{L}(L_{n_1}(x;q,y)\cdots L_{n_k}(x;q,y))$ and induction, they proved Theorem~\ref{thm:Main}.
The purpose of this paper is to give a proof of Theorem~\ref{thm:Main} by
constructing a sign-reversing involution. Our fundamental combinatorial objects
are matchings instead of permutations.

The remainder of this paper is organized as follows. In
Section~\ref{sec:Laguerre} we give basic definitions and
combinatorial interpretations for $L_n(x;q,y)$ and $\mu_n(q,y)$ using matchings
and perfect matchings. In Section~\ref{sec:line-coeff-sign} we give a
combinatorial model for the linearization coefficient in terms of marked perfect
matchings. We then construct a sign-reversing involution on marked perfect
matchings. Section~\ref{sec:proofinvolution} is devoted to showing that our map in
Section~\ref{sec:line-coeff-sign} is indeed a sign-reversing involution that preserves
the desired weights on marked perfect matchings.
In the final section we discuss future work. 

An extended abstract of this paper will appear in the Proceedings of the 32nd
Conference on Formal Power Series and Algebraic Combinatorics.


\section{$q$-Laguerre polynomials and their moments} \label{sec:Laguerre}

In this section we give combinatorial interpretations for the $q$-Laguerre
polynomials \linebreak[4] $L_n(x;q,y)$ and their moments $\mu_n(q,y)$ using matchings and
perfect matchings. The results in this section generalize the combinatorial
models for Laguerre polynomials and their moments due to Viennot
\cite[Ch.~6]{ViennotLN}. We start with basic definitions.

\begin{defn}
  Let $K_{n,n}$ be the complete bipartite graph with $2n$ vertices, i.e., the
  graph with vertex set $\{1,2,\dots,n,\overline{1},\overline{2}, \dots,
  \overline{n}\}$ and edge set $\{(i,\overline{j}): 1\le i,j\le n\}$. A
  \emph{matching} of degree $n$ is a subgraph $\pi$ of $K_{n,n}$ such that $\pi$
  contains every vertex of $K_{n,n}$ and no two distinct edges of $\pi$ have
  common vertices. A matching $\pi$ of degree $n$ is called a \emph{perfect
    matching} if $\pi$ has exactly $n$ edges. Denote the set of all matchings
  (respectively, perfect matchings) of degree $n$ by $\M_n$ (respectively,
  $\PM_n$). For $\pi\in \M_n$, we denote by $E(\pi)$ the set of edges in $\pi$
  and let $e(\pi)=|E(\pi)|$.
\end{defn}

We visualize a matching $\pi$ of degree $n$ by placing the vertices
$1,2,\dots,n$ in the upper row and the vertices
$\overline{1},\overline{2},\dots,\overline{n}$ in the lower row as shown in
Figure~\ref{fig:matching}. We call $1,2,\dots,n$ the \emph{upper vertices} and
$\overline{1},\overline{2},\dots,\overline{n}$ the \emph{lower vertices} of
$\pi$. If there is no possible confusion, we will simply write $j$ instead of
$\overline{j}$. For example, since every edge of a matching is of the form
$(i,\overline{j})$, we will also write this edge as $(i,j)$.

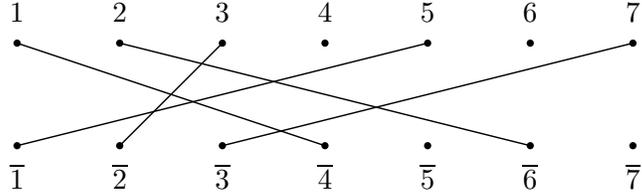
\begin{figure}[H]
    \centering
    \resizebox{0.6\linewidth}{!}{
    \begin{tikzpicture}
    \tikzmath{\m = 0.15;\mh=0.3;\Mv=0.2;\Mh=0.45;\Mp=8.5;};


    \draw[unmarked] (1,1) -- (4,0);
    \draw[unmarked] (2,1) -- (6,0);
    \draw[unmarked] (3,1) -- (2,0);
    \draw[unmarked] (5,1) -- (1,0);
    \draw[unmarked] (7,1) -- (3,0);


\foreach \x   in {1,...,7}
{
    \fill (\x,1) circle (1pt);
    \fill (\x,0) circle (1pt);
    \node[scale=0.8] at (\x,1.3) {$\x$};
    \node[scale=0.8] at (\x,-0.3) {$\overline{\x}$};
    }
\end{tikzpicture}
	}   
    \caption{A matching $\pi$ of degree 7, which is not a perfect matching.} \label{fig:matching}
\end{figure}

For $\pi\in \M_n$, if $(i,j)\in \pi$, we denote $\pi(i) = j$ and
$e_i=(i,\pi(i))$. For example, if $\pi$ is the matching in
Figure~\ref{fig:matching}, then $\pi(1)=4$, $\pi(3)=2$ and $e_1=(1,4)$, $e_3=(3,2)$.
An upper vertex $i$ of $\pi$ is said to be \emph{unmatched} if there is no edge
of the form $(i,j)$. Similarly, a lower vertex $j$ of $\pi$ is \emph{unmatched}
if there is no edge of the form $(i,j)$. Note that if $\pi\in\PM_n$, there are
no unmatched vertices and we can identify $\pi$ with the permutation $\sigma\in
S_n$ given by $\sigma(i)=\pi(i)$ for all $i\in[n]$. We will often use this
identification in this paper.

Let $\pi\in \PM_n$. An edge $e=(i,\pi(i))$ of $\pi$ is called a \emph{weak
  excedance} if $i\le \pi(i)$. A pair $(e, e'')$ of edges $e=(i,\pi(i))$ and
$e'=(j,\pi(j))$ is said to be \emph{overlapping} if $i<j\le\pi(i)<\pi(j)$ or
$\pi(i)<\pi(j)<i<j$. Let $\wex(\pi)$ and $\ov(\pi)$ denote the number of weak
excedances and overlapping pairs of $\pi$. In other words,
\begin{align*}
\wex(\pi) &= | \{ i\in [n] : \pi(i)\ge i\} | \quad\mbox{and} \\
\ov(\pi) &= | \{ (i, j) \in [n]\times [n] : i<j\le\pi(i)<\pi(j)\mbox{ or }\pi(j)<\pi(i)<j<i \}|.
\end{align*}
By the identification of $\PM_n$ and $S_n$ we can rewrite \eqref{eq:mu1} as
follows:
\begin{equation} \label{eq:general_moment1}
\mu_n(q,y) = \LL(x^n) = \sum_{\pi\in \PM_n} y^{\wex(\pi)} q^{\ov(\pi)}.
\end{equation}

For the remainder of this section we will find a combinatorial model for
$L_n(x;q,y)$ in Theorem~\ref{thm:Laguerre} and give yet another expression for
$\mu_n(q,y)$ in \eqref{eq:general_moment2}. To do this, we define some
statistics for matchings. Given a matching $\pi\in \M_n$, let
$P=(B_1,\dots,B_l)$ be the unique ordered set partition of the upper vertices of
$\pi$ satisfying the following conditions:
\begin{itemize}
\item Each block $B_r$ consists of consecutive elements. In other words, $B_r$ is of the form $B_r=\left\{i,i+1,\dots, j\right\}$.
\item For each $i\in[n]$, $i$ is the largest element in some block $B_r$ if and
  only if $i$ is an unmatched vertex or $i=n$.
\end{itemize}
We define the \emph{upper block index} $\bind^U_\pi(i)$ of a vertex $i$ to be
the integer $r$ such that $i\in B_r$. Note that $\bind^U_\pi(i)$ is equal to one
more than the number of unmatched vertices appearing before $i$ in the upper
row. The \emph{lower block index} $\bind^L_\pi(i)$ is defined similarly by
considering the ordered set partition of the lower vertices of $\pi$.

\begin{defn}
\label{defn:stat_Mn}
For a matching $\pi\in \M_n$, the \emph{block difference} $\bdiff_\pi(e)$ of an
edge $e=(i,\pi(i))$ is the difference between $\bind_\pi^L(\pi(i))$ and
$\bind_\pi^U(i)$, that is,
\[
  \bdiff_\pi(e) = \bind_\pi^L(\pi(i)) -\bind_\pi^U(i).
\]
An edge $e\in E(\pi)$ is called a \emph{block weak excedance} if
$\bdiff_\pi(e)\ge 0$. Denote the number of block weak excedances in $\pi$ by
$\bwex(\pi)$. The \emph{block weight} $\bwt(\pi)$ of $\pi\in \M_n$ is defined by
\[
\bwt(\pi) = \sum_{\bdiff_\pi(e)\ge 0} \bdiff_\pi(e) + \sum_{\bdiff_\pi(e) < 0} (-\bdiff_\pi(e)-1).
\]
A \emph{crossing} of $\pi$ is a pair $(e,e')$ of edges $e=(i,\pi(i))$ and
$e'=(j,\pi(j))$ in $\pi$ such that $i<j$ and $\pi(i)>\pi(j)$. The number of
crossings of $\pi$ is denoted by $\cross(\pi)$.
\end{defn}

We note that the notion of crossing for a matching $\pi\in\M_n$ is different
from that for a permutation $\sigma\in S_n$. If $\pi\in\PM_n$ corresponds to
$\sigma\in S_n$ using the identification, we have $\CR(\sigma)=\ov(\pi)$ but 
$\CR(\sigma)\ne\cross(\pi)$. A crossing of $\pi\in \M_n$ can be understood as a
pair of edges that intersect in the visualization of $\pi$.

\begin{exam}
  Let $\pi$ be the matching in Figure~\ref{fig:matching}. Then the ordered set
  partition for the upper row is $( \{1,2,3,4\}, \{5,6\}, \{7\})$ and the
  ordered set partition for the lower row is $( \{1,2,3,4,5\}, \{6,7\} )$. Let
  $e=(7,\overline{3})$. The block indices of its two endpoints are
  $\bind^U_\pi(7) = 3$ and $\bind^L_\pi(3) = 1$, so we have $\bdiff_\pi(e) =
  -2$. The number of block weak excedances in $\pi$ is $\bwex(\pi)=3$, the block
  weight of $\pi$ is $\bwt(\pi) = 0$, and the number of crossings of $\pi$ is
  $\cross(\pi)=7$.

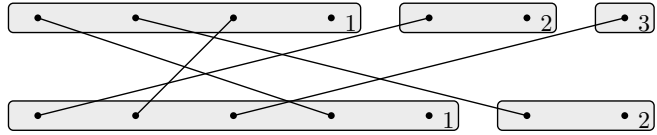
\begin{figure}[H]
    \centering
    \resizebox{0.6\linewidth}{!}{
    \begin{tikzpicture}
    \pgfsetcornersarced{\pgfpoint{0.5mm}{0.5mm}}
    \tikzmath{\m = 0.15;\mh=0.3;\Mv=0.2;\Mh=0.45;\Mp=8.5;};

    \filldraw[block] (1-\mh,1-\m) rectangle (4+\mh,1+\m);
    \filldraw[block] (5-\mh,1-\m) rectangle (6+\mh,1+\m);
    \filldraw[block] (7-\mh,1-\m) rectangle (7+\mh,1+\m);
    \filldraw[block] (1-\mh,0-\m) rectangle (5+\mh,0+\m);
    \filldraw[block] (6-\mh,0-\m) rectangle (7+\mh,0+\m);

    \draw[unmarked] (1,1) -- (4,0);
    \draw[unmarked] (2,1) -- (6,0);
    \draw[unmarked] (3,1) -- (2,0);
    \draw[unmarked] (5,1) -- (1,0);
    \draw[unmarked] (7,1) -- (3,0);


\foreach \x   in {1,...,7}
{
    \fill (\x,1) circle (1pt);
    \fill (\x,0) circle (1pt);
    }

    \node[scale=0.7] at (4+0.2,1-0.05) {1};
    \node[scale=0.7] at (6+0.2,1-0.05) {2};
    \node[scale=0.7] at (7+0.2,1-0.05) {3};
    \node[scale=0.7] at (5+0.2,0-0.05) {1};
    \node[scale=0.7] at (7+0.2,0-0.05) {2};
\end{tikzpicture}
	}
    \caption{A matching $\pi$ with its blocks. The block numbers are shown.}  \label{fig:block_index}
\end{figure}

\end{exam}

We are now ready to express the $q$-Laguerre polynomials combinatorially.

\begin{thm} \label{thm:Laguerre}
For $n\ge 0$, we have
\begin{equation} \label{eq:general_Laguerre_combi}
L_n(x;q,y) = \sum_{\pi\in \M_n} (-1)^{\ee(\pi)} y^{\bwex(\pi)} q^{\bwt(\pi)+\cross(\pi)} x^{n-\ee(\pi)}.
\end{equation}
\end{thm}

\begin{exam}
There are 7 matchings of degree $2$ as shown in Figure~\ref{fig:7}. 

\begin{figure}[H]
    \centering
    
    \begin{tikzpicture}

\foreach \x[count=\xi]   in 
{-1,1,2,-1}
{
    \fill (2.5*\xi+2,1) circle (1pt);
    \fill (2.5*\xi+1,1) circle (1pt);
    \fill (2.5*\xi+2,0) circle (1pt);
    \fill (2.5*\xi+1,0) circle (1pt);
    \ifthenelse{\x>0}{\draw[unmarked] (2.5*\xi+1,1) -- (2.5*\xi+\x,0);}{}
}
\foreach \x[count=\xi]   in 
{-1,-1,-1,1}
{
    \ifthenelse{\x>0}{\draw[unmarked] (2.5*\xi+2,1) -- (2.5*\xi+\x,0);}{}
}

\foreach \x[count=\xi]   in 
{-1,1,2}
{
    \fill (2.5*\xi+3.25,1-2.5) circle (1pt);
    \fill (2.5*\xi+2.25,1-2.5) circle (1pt);
    \fill (2.5*\xi+3.25,0-2.5) circle (1pt);
    \fill (2.5*\xi+2.25,0-2.5) circle (1pt);
    \ifthenelse{\x>0}{\draw[unmarked] (2.5*\xi+2.25,1-2.5) -- (2.5*\xi+1.25+\x,0-2.5);}{}
}
\foreach \x[count=\xi]   in 
{2,2,1}
{
    \ifthenelse{\x>0}{\draw[unmarked] (2.5*\xi+3.25,1-2.5) -- (2.5*\xi+1.25+\x,0-2.5);}{}
}

    \node[scale=1] at (1*2.5+1.5,-0.5) {$x^2$};
    \node[scale=1] at (2*2.5+1.5,-0.5) {$-xy$};
    \node[scale=1] at (3*2.5+1.5,-0.5) {$-xyq$};
    \node[scale=1] at (4*2.5+1.5,-0.5) {$-x$};
    \node[scale=1] at (1*2.5+1.5+1.25,-0.5-2.5) {$-xy$};
    \node[scale=1] at (2*2.5+1.5+1.25,-0.5-2.5) {$y^2$};
    \node[scale=1] at (3*2.5+1.5+1.25,-0.5-2.5) {$y^2q$};
\end{tikzpicture}
    \caption{The matchings of degree $2$ and their corresponding terms.}
    \label{fig:7}
\end{figure}
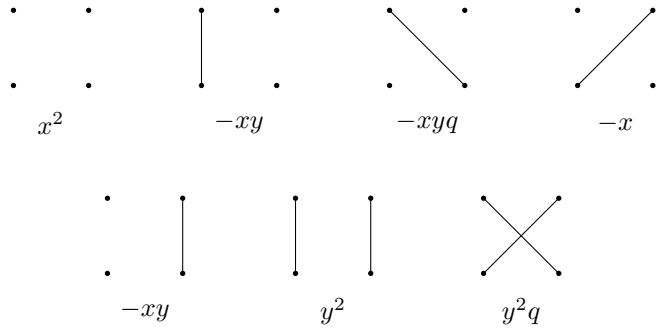
Then by Theorem~\ref{thm:Laguerre}, we have
\[
L_2 (x;q,y) = x^2-(yq + 2y + 1)x + y^2 + y^2q.
\]
\end{exam}

\begin{proof}[Proof of Theorem~\ref{thm:Laguerre}]
  The proof is by induction on $n$. The cases for $n=0,1$ are easy to check. For
  $n\ge 2$ we will show that the right hand side of
  \eqref{eq:general_Laguerre_combi} satisfies the three-term recurrence
  \eqref{eq:general_recur}, which we recall here:
  \begin{equation}
    \label{eq:1}
L_{n+1}(x;q,y) = (x-y[n+1]_q-[n]_q) L_{n}(x;q,y) - y[n]^2_q L_{n-1}(x;q,y).
  \end{equation}
For each matching $\pi\in \M_{n+1}$ there are three cases as follows.

\begin{description}
\item[Case 1] Two vertices $n+1$ and $\overline{n+1}$ are both unmatched. Let
  $\pi'\in \M_n$ be the matching obtained from $\pi$ by deleting the last vertex
  in each row. Clearly all statistics but the number of unmatched vertices of
  $\pi$ and $\pi'$ are equal. Then this case contributes $xL_n(x;q,y)$ to the
  right-hand side of \eqref{eq:1}.

\item[Case 2] The vertex $\overline{n+1}$ is matched to some vertex $i$, i.e.,
  there is an edge $e_i=(i, \overline{n+1})$ $\in E(\pi)$. Let $\pi'$ be the matching
  obtained from $\pi$ by deleting $e_i$ and its end vertices and we regard
  $\pi'$ as a matching in $\M_n$. Since the deleted vertex $i$ is matched in
  $\pi$, the block indices of vertices of $\pi$ and $\pi'$ are equal, so are the
  block differences. That is, $\bdiff_\pi(e) = \bdiff_{\pi'}(e)$ for $e\in
  E(\pi)\setminus \{e_i\}$. Since the number of block weak excedances and the
  block weight of $\pi$ depend only on the block differences, we only need to
  consider the contribution of $e_i$ to $\bwex(\pi)$ and $\bwt(\pi)$. The lower
  block index $\bind^L_\pi(\overline{n+1})$ is one more than the number of
  unmatched vertices in the lower row, so $\bind^L_\pi(\overline{n+1}) =
  n+2-\ee(\pi)$. Then $e_i$ is automatically a block weak excedance, so
  $\bwex(\pi) = \bwex(\pi')+1$.
  To consider the block weight, let $m$ be the number of matched upper vertices $j$ such that $i<j$. It is clear that an edge $e_j$ crosses $e_i$ if and only if $i<j$, and hence $\cross(\pi) = \cross(\pi')+m$. It is easy to check that $\bind^U_\pi(i)=i+1-\ee(\pi)+m$, so $\bdiff_\pi(e_i) = n+1-i-m$ and $\bwt(\pi) = \bwt(\pi')+n+1-i-m$. Thus Case 2 corresponds to the term $\sum_{i=1}^{n+1} (-yq^{n+1-i} L_n(x;q,y)) = -y [n+1]_q L_n(x;q,y)$.
  
\item[Case 3] The vertex $\overline{n+1}$ is unmatched and the vertex $n+1$ is
  matched to some vertex $\overline{i}$ where $i\le n$. This case is similar to
  Case 2, except that the edge $(n+1,\overline{i})$ is not a block weak
  excedance. Letting $\widetilde{\M}_n$ be the set of matchings in $\M_n$ such
  that $\overline{n}$ is unmatched and
\[
\widetilde{L}_n(x;q,y) \coloneqq \sum_{\pi\in \widetilde{\M}_n} (-1)^{\ee(\pi)} y^{\bwex(\pi)} q^{\bwt(\pi)+\cross(\pi)} x^{n-\ee(\pi)},
\]
we obtain that Case 3 contributes $-[n]_q \widetilde{L}_n(x;q,y)$.
\end{description}
From Cases 1, 2 and 3, we have
\[
L_{n+1}(x;q,y) = (x-y[n+1]_q) L_n(x;q,y) - [n]_q \widetilde{L}_n(x;q,y).
\]
Comparing this with \eqref{eq:1}, it is enough to show that
\[
L_n(x;q,y) = \widetilde{L}_n(x;q,y) - y[n]_q L_{n-1}(x;q,y).
\]
By the same argument in Case 2, the second term (including the negative sign) in the right-hand side of the above equation is equal to
\[
\sum_{\pi\in \M_n\setminus \widetilde{M}_n} (-1)^{\ee(\pi)} y^{\bwex(\pi)} q^{\bwt(\pi)+\cross(\pi)} x^{n-\ee(\pi)},
\]
then the proof follows.
\end{proof}

Now we modify the combinatorial expression \eqref{eq:general_moment1} for the
moment $\mu_n(q,y)$ so that the new expression is more suitable for our
approach. For $\pi\in \PM_n$, the \emph{weight} $\wt(\pi)$ of $\pi$ is defined
by
\[
\wt(\pi)= \sum\limits_{\pi(i)\ge i}(\pi(i)-i)+\sum\limits_{\pi(i)<i}(i-\pi(i)-1).
\]
In fact, this definition is obtained from the definition of the block weight by replacing block differences $\bdiff_\pi(e)$ by $\pi(i)-i$. The following lemma gives a relation between $\ov(\pi)$, $\wt(\pi)$ and $\cross(\pi)$.
\begin{lem} \label{lem:ov_wt_cr}
For $\pi\in \PM_n$, $\ov(\pi)=\wt(\pi)-\cross(\pi)$.
\end{lem}
\begin{proof}
We prove that $\wt(\pi)=\ov(\pi)+\cross(\pi)$.
By the definition of weight,
\begin{align*}
\wt(\pi)
    &=\sum\limits_{\pi(i)\ge i}(\pi(i)-i)+\sum\limits_{\pi(i)<i}(i-\pi(i)-1)\\
    &=\sum\limits_{\pi(i)\ge i}|\{j:i\le\pi(j)<\pi(i)\}|+\sum\limits_{\pi(i)<i}|\{j:\pi(i)<\pi(j)<i\}|\\
    &=\sum\limits_{\pi(i)\ge i}|\{j:j<i\le\pi(j)<\pi(i)\}|+\sum\limits_{\pi(i)\ge i}|\{j:i<j\mbox{ and }i\le\pi(j)<\pi(i)\}|\\
    &\qquad +\sum\limits_{\pi(i)<i}|\{j:\pi(i)<\pi(j)<i<j\}|+
    \sum\limits_{\pi(i)<i}|\{j:j<i\mbox{ and }\pi(i)<\pi(j)<i\}|.
\end{align*}
On the right-hand side of the last equation, it is clear that the sum of the first and third summands is equal to $\ov(\pi)$.

On the other hand, the weight of $\pi$ can also be expressed as
\begin{align*}
\wt(\pi)
    &=\sum\limits_{\pi(i)\ge i}(\pi(i)-i)+\sum\limits_{\pi(i)<i}(i-\pi(i)-1)\\
    &=\sum\limits_{\pi(i)\ge i}|\{j:i<j\le\pi(i)\}|+\sum\limits_{\pi(i)<i}|\{j:\pi(i)<j<i\}|\\
    &=\sum\limits_{\pi(i)\ge i}|\{j:i<j\le\pi(i)<\pi(j)\}|+\sum\limits_{\pi(i)\ge i}|\{j:\pi(j)<\pi(i)\mbox{ and }i<j\le\pi(i)\}|\\
    &\qquad+\sum\limits_{\pi(i)<i}|\{j:\pi(j)<\pi(i)<j<i\}|+
    \sum\limits_{\pi(i)<i}|\{j:\pi(i)<\pi(j)\mbox{ and }\pi(i)<j<i\}|.
\end{align*}
Similarly, in the right-hand side of the last equation, it is clear that the sum of the first and third summands is equal to $\ov(\pi)$. Thus, with a slight change of variables, it is enough to show that
\begin{align*}
2\cross(\pi)
    &=\sum\limits_{\pi(i)\ge i}|\{j:i<j\mbox{ and } i\le\pi(j)<\pi(i)\}|\\&\qquad+\sum\limits_{\pi(i)<i}|\{j:j<i\mbox{ and }\pi(i)<\pi(j)<i\}|\\
    &\qquad+\sum\limits_{\pi(i)\ge i}|\{j:\pi(j)<\pi(i)\mbox{ and }i<j\le\pi(i)\}|\\&\qquad+\sum\limits_{\pi(i)<i}|\{j:\pi(i)<\pi(j)\mbox{ and }\pi(i)<j<i\}|\\
\end{align*}
\begin{align*}
    &=|\{(i,j):i<j, \pi(i)>\pi(j) \mbox{ and } i\le\pi(j)\}|\\&\qquad+|\{(i,j):i<j, \pi(i)>\pi(j)\mbox{ and }j>\pi(i)\}|\\
    &\qquad+|\{(i,j):i<j, \pi(i)>\pi(j)\mbox{ and }j\le\pi(i)\}|\\&\qquad+|\{(i,j):i<j,\pi(i)>\pi(j)\mbox{ and }i>\pi(j)\}|.
\end{align*}
One can see that each crossing of $\pi$ is counted twice in the right-hand side
of the above equation. To be precise, a pair $(i,j)$ of integers such that
$(e_i,e_j)$ is a crossing of $\pi$ is counted once either in the first or 
last summand depending on the sign of $i-\pi(j)$, and counted once again either
in the second or third summand depending on the sign of $j-\pi(i)$. This
completes the proof.
\end{proof}

By Lemma~\ref{lem:ov_wt_cr} we can rewrite the moment $\mu_n(q,y)$ using $\wt(\pi)$ and $\cross(\pi)$ instead of $\ov(\pi)$:
\begin{equation} \label{eq:general_moment2}
\mu_n(q,y) = \sum_{\pi\in \PM_n} y^{\wex(\pi)} q^{\wt(\pi)-\cross(\pi)}.
\end{equation}
In the next section we will use Theorem~\ref{thm:Laguerre} and
\eqref{eq:general_moment2} to give a combinatorial meaning to the linearization
coefficients of $q$-Laguerre polynomials.


\section{Linearization coefficients and a sign-reversing involution}
\label{sec:line-coeff-sign}
\subsection{A combinatorial interpretation of linearization coefficients}

In this section we give a combinatorial interpretation of the
linearization coefficient $C(n_1,\dots,n_k)\coloneqq\LL(L_{n_1}\cdots L_{n_k})$
of the $q$-Laguerre polynomials $L_{n}=L_{n}(x;q,y)$.
 First we recall the
expression of $L_n$ in terms of matchings in Theorem \ref{thm:Laguerre}:
\begin{equation}
\label{eq:singleL}
L_n=\sum_{\pi \in \M_n}(-1)^{\ee(\pi)}y^{\bwex(\pi)}q^{\bwt(\pi)+\cross(\pi)}x^{n-\ee(\pi)}.
\end{equation}
To give a description of the product $L_{n_1}\cdots L_{n_k}$, we embed
$\M_{n_1}\times\cdots\times \M_{n_k}$ in $\M_{N}$, where
$N=\sum_{i=1}^{k}n_i$, by horizontally concatenating the $k$ matchings
$\pi_1,\dots,\pi_k$ for each $(\pi_1,\dots,\pi_k)\in
\M_{n_1}\times\dots\times\M_{n_k}$. Let $\M_{n_1,\dots,n_k}\subset \M_N$ denote
the embedded image of $\M_{n_1}\times\cdots\times \M_{n_k}$.

Let $\pi\in \M_N$. We say that an edge $(i,\pi(i))$ of $\pi$ is \emph{homogeneous} with
respect to $(n_1,\dots,n_k)$ if 
\[
n_1+\dots+n_{r-1}+1 \le i,\pi(i) \le n_1+\dots+n_r,
\]
for some $1\le r\le k$, and \emph{inhomogeneous} otherwise. For
simplicity, we omit the expression `with respect to $(n_1,\dots,n_k)$' when
there is no confusion. Note that $\M_{n_1,\dots,n_k}$ is the set of matchings in
$\M_N$ such that every edge is homogeneous. We will write $E^H(\pi)$ for the set
of homogeneous edges of $\pi$.

Note that if $\pi\in\M_{n_1,\dots,n_k}$ is the concatenation of
$\pi_1,\dots,\pi_k$, then each statistic in \eqref{eq:singleL} satisfies the
relation $\operatorname{stat}(\pi)=\sum_{i=1}^{k}\operatorname{stat}(\pi_i)$.
Thus the product $L_{n_1}\cdots L_{n_k}$ is written as
\begin{equation}\label{eq:product}
L_{n_1}\cdots L_{n_k}=\sum_{\pi \in \M_{n_1,\dots,n_k}}(-1)^{\ee(\pi)}y^{\bwex(\pi)}q^{\bwt(\pi)+\cross(\pi)}x^{N-\ee(\pi)}.
\end{equation}
Applying $\LL$ to \eqref{eq:product}, we have
\[
  \LL\left(L_{n_1}\cdots L_{n_k}\right)=\sum_{\pi\in
    \M_{n_1,\dots,n_k}}(-1)^{\ee(\pi)}y^{\bwex(\pi)}q^{\bwt(\pi)+\cross(\pi)}\LL\left(x^{N-\ee(\pi)}\right).
\]
Here we recall the formula of the $n$th moment in \eqref{eq:general_moment2}:
$$
\mu_n(q,y)=\LL(x^n)=\sum_{\pi \in \PM_{n}}y^{\wex(\pi)}q^{\wt(\pi)-\cross(\pi)}.
$$
 Note that $N-\ee(\pi)$, the power of $x$ in \eqref{eq:product}, represents the number of unmatched vertices in the upper (or lower) row, or equivalently, the number of edges we need to add to make it a perfect matching. Thus, applying the functional $\LL$ to $x^{N-\ee(\pi)}$ is interpreted as summing up all possible ways to complete $\pi$ into a perfect matching, by adding edges on the unmatched vertices, allowing inhomogeneous edges.
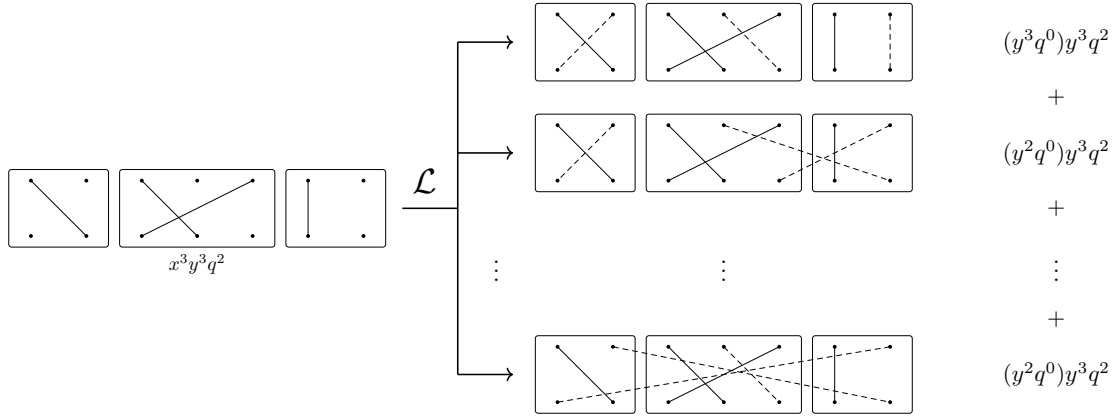
\begin{figure}[H]
    \resizebox{\linewidth}{!}{
    \begin{tikzpicture}
\pgfsetcornersarced{\pgfpoint{0.5mm}{0.5mm}}
    \tikzmath{\Mv=0.2;\Mh=0.4;\Mp=8.5;};
    
    \draw (1-\Mh-1,0-\Mv-3) rectangle (2+\Mh-1,1+\Mv-3);
    \draw (3-\Mh-1,0-\Mv-3) rectangle (5+\Mh-1,1+\Mv-3);
    \draw (6-\Mh-1,0-\Mv-3) rectangle (7+\Mh-1,1+\Mv-3);
    

    \draw[unmarked] (1-1,1-3) -- (2-1,0-3);
    \draw[unmarked] (3-1,1-3) -- (4-1,0-3);
    \draw[unmarked] (5-1,1-3) -- (3-1,0-3);
    \draw[unmarked] (6-1,1-3) -- (6-1,0-3);

\foreach \x  in {1,...,7}
{
    \fill (\x-1,1-3) circle (1pt);
    \fill (\x-1,0-3) circle (1pt);
    }
    \node[scale=1] at (4-1,-0.5-3) {$x^3y^3q^2$};
    \draw[->,thick] (7.7,0.5) -- (8.7,0.5);
    \draw[->,thick] (7.7,0.5-2) -- (8.7,0.5-2);
    \draw[->,thick] (7.7,0.5-6) -- (8.7,0.5-6);
    \draw[thick] (7.7,0.5) -- (7.7,0.5-6);
    \draw[thick] (6.7,-2.5) -- (7.7,-2.5);
    \node[scale=1.8] at (7.1,-2) {$\LL$};

    
    \draw (1-\Mh+\Mp,0-\Mv) rectangle (2+\Mh+\Mp,1+\Mv);
    \draw (3-\Mh+\Mp,0-\Mv) rectangle (5+\Mh+\Mp,1+\Mv);
    \draw (6-\Mh+\Mp,0-\Mv) rectangle (7+\Mh+\Mp,1+\Mv);
    

    \draw[unmarked] (1+\Mp,1) -- (2+\Mp,0);
    \draw[unmarked] (3+\Mp,1) -- (4+\Mp,0);
    \draw[unmarked] (5+\Mp,1) -- (3+\Mp,0);
    \draw[unmarked] (6+\Mp,1) -- (6+\Mp,0);
    
    \draw[marked] (2+\Mp,1) -- (1+\Mp,0);
    \draw[marked] (4+\Mp,1) -- (5+\Mp,0);
    \draw[marked] (7+\Mp,1) -- (7+\Mp,0);

\foreach \x  in {1,...,7}
{
    \fill (\x+\Mp,1) circle (1pt);
    \fill (\x+\Mp,0) circle (1pt);
    }
\node[scale=1.2] at (10+\Mp,0.5) {$(y^3q^0)y^3q^2$};
\node[scale=1.2] at (10+\Mp,-0.5) {$+$};
    
    \draw (1-\Mh+\Mp,0-\Mv-2) rectangle (2+\Mh+\Mp,1+\Mv-2);
    \draw (3-\Mh+\Mp,0-\Mv-2) rectangle (5+\Mh+\Mp,1+\Mv-2);
    \draw (6-\Mh+\Mp,0-\Mv-2) rectangle (7+\Mh+\Mp,1+\Mv-2);
    

    \draw[unmarked] (1+\Mp,1-2) -- (2+\Mp,0-2);
    \draw[unmarked] (3+\Mp,1-2) -- (4+\Mp,0-2);
    \draw[unmarked] (5+\Mp,1-2) -- (3+\Mp,0-2);
    \draw[unmarked] (6+\Mp,1-2) -- (6+\Mp,0-2);
    
    \draw[marked] (2+\Mp,1-2) -- (1+\Mp,0-2);
    \draw[marked] (4+\Mp,1-2) -- (7+\Mp,0-2);
    \draw[marked] (7+\Mp,1-2) -- (5+\Mp,0-2);

\foreach \x  in {1,...,7}
{
    \fill (\x+\Mp,1-2) circle (1pt);
    \fill (\x+\Mp,0-2) circle (1pt);
    }
\node[scale=1.2] at (10+\Mp,0.5-2) {$(y^2q^0)y^3q^2$};
\node[scale=1.2] at (10+\Mp,-0.5-2) {$+$};
\node[scale=1.2] at (8.4,0.5-4) {$\vdots$};
\node[scale=1.2] at (4+\Mp,0.5-4) {$\vdots$};
\node[scale=1.2] at (10+\Mp,0.5-4) {$\vdots$};
\node[scale=1.2] at (10+\Mp,-0.5-4) {$+$};
    
    \draw (1-\Mh+\Mp,0-\Mv-6) rectangle (2+\Mh+\Mp,1+\Mv-6);
    \draw (3-\Mh+\Mp,0-\Mv-6) rectangle (5+\Mh+\Mp,1+\Mv-6);
    \draw (6-\Mh+\Mp,0-\Mv-6) rectangle (7+\Mh+\Mp,1+\Mv-6);
    

    \draw[unmarked] (1+\Mp,1-6) -- (2+\Mp,0-6);
    \draw[unmarked] (3+\Mp,1-6) -- (4+\Mp,0-6);
    \draw[unmarked] (5+\Mp,1-6) -- (3+\Mp,0-6);
    \draw[unmarked] (6+\Mp,1-6) -- (6+\Mp,0-6);
    
    \draw[marked] (2+\Mp,1-6) -- (7+\Mp,0-6);
    \draw[marked] (4+\Mp,1-6) -- (5+\Mp,0-6);
    \draw[marked] (7+\Mp,1-6) -- (1+\Mp,0-6);

\foreach \x  in {1,...,7}
{
    \fill (\x+\Mp,1-6) circle (1pt);
    \fill (\x+\Mp,0-6) circle (1pt);
    }
\node[scale=1.2] at (10+\Mp,0.5-6) {$(y^2q^0)y^3q^2$};
\end{tikzpicture}
}

 \captionsetup{format=hang}
    \caption{An example of applying $\LL$ to a term $x^3y^3q^2$ in the product $L_2L_3L_2$. There are 3!=6 terms in $\LL(x^3)$ corresponding to all possible completions of the original matching. }
    \label{fig:linearfunctional_ex}
\end{figure}
\begin{exam}
Figure \ref{fig:linearfunctional_ex} describes an example of the application of
$\LL$. The matching on the left side represents a term $x^3y^3q^2$ in
$L_2L_3L_2$, which is the product of three terms $-xyq$, $xyq$ and $-xy$ in $L_2$, $L_3$ and $L_2$, respectively. Applying $\LL$ gives an equation $$
\left(\sum_{\pi \in \PM_{3}}y^{\wex(\pi)}q^{\wt(\pi)-\cross(\pi)}\right)y^3q^2,
$$
where each summand corresponds to a way to add edges to remaining vertices, represented in dashed lines.
\end{exam}
 In order to describe the expansion of $\LL(L_{n_1}\cdots L_{n_k})$, we introduce a perfect matching model containing the information of which edges are newly added by applying $\LL$.
 Let $\PM_{n_1,\dots,n_k}^{*}$ be the set of pairs $\mm=(\pi,S)$ such that
 \begin{itemize}
    \item $\pi\in \M_{N}$ is a perfect matching of degree $N=\sum_{i=1}^{k}n_i$,
    \item $S$ is a subset of edges in $\pi$, which contains all inhomogeneous edges of $\pi$, i.e., $E(\pi)\setminus E^H(\pi)\subseteq S$.
\end{itemize}
We call an element $\mm=(\pi,S)$ of $\PM_{n_1,\dots,n_k}^{*}$ a \emph{marked
  perfect matching}. An edge $e$ of $\pi$ is said to be \emph{marked} if $e\in
S$. In other words, $S$ is the set of marked edges. With marks on edges, we can
distinguish new edges added by applying $\LL$ from the original edges from
$L_{n_1}\cdots L_{n_k}$. The condition $E(\pi)\setminus E^H(\pi)\subseteq S$ is
needed since inhomogeneous edges cannot be present in the original matching
coming from $L_{n_1}\cdots L_{n_k}$.

Now we give a bijective correspondence between $\PM^{*}_{n_1,\dots,n_k}$ and the terms in the expansion of
$\LL(L_{n_1}\cdots L_{n_k})$. To do this, we
extend our former definitions of statistics on $\M_n$ and $\PM_n$ to marked
perfect matchings. In detail, we consider the decomposition of $\mm$ into
unmarked and marked portions. For $\mm=(\pi,S)\in \PM^{*}_{n_1,\dots,n_k}$,
define $\pi\setminus S$ and $\pi|_S$ as follows:
\begin{itemize}
\item $\pi\setminus S$ (\emph{unmarked portion} of $\mm$) is the matching in
  $\M_{n_1,\dots,n_k}$ with $n_1+\cdots+n_k-|S|$ edges obtained from $\pi$ by
  deleting the $|S|$ marked edges but leaving their incident vertices not
  deleted.
\item $\pi|_S$ (\emph{marked portion} of $\mm$) is the perfect matching in
  $\PM_{|S|}$ obtained from $\pi$ by deleting all unmarked edges and their
  adjacent vertices.
\end{itemize}

\begin{defn}
\label{defn:stat_PMstar}
For $\mm=(\pi,S)\in \PM^{*}_{n_1,\dots,n_k}$, define statistics $\ee(\mm),\bwex(\mm),\cross(\mm)$ and $\wt(\mm)$ as follows:
\begin{align*}
    \ee(\mm)&=\ee(\pi\setminus S),&
    \bwex(\mm)&=\bwex(\pi\setminus S)+\wex(\pi|_S),\\
    \cross(\mm)&=\cross(\pi\setminus S)-\cross(\pi|_S),&
    \wt(\mm)&=\bwt(\pi\setminus S)+\wt(\pi|_S).
\end{align*}
\end{defn}
\begin{remark}
\label{remark:stat_explained}
Indeed, the notions of $\bwex$ and $\wt$ in Definition \ref{defn:stat_PMstar} is still compatible with those of block index and block difference we defined earlier on a matching in $\M_n$. The only difference is that the blocks are separated by the vertices incident to marked edges, instead of unmatched ones. More precisely, for a marked perfect matching $\mm=(\pi,S)\in \PM_{n_1,\dots,n_k}^{*}$, let $P=(B_1,\dots,B_l)$ be the unique ordered set partition of upper vertices satisfying the following conditions:
\begin{itemize}
    \item Each $B_r$ consists of consecutive elements. In other words, $B_r$ is of the form $B_r=\left\{i,i+1,\dots, j\right\}$.
    \item For $i\in\left\{1,\dots,n_1+\cdots+n_k\right\}$, $i$ is the largest element in some block $B_r$ if and only if $i$ is incident to a marked edge or $i=n_1+\cdots+n_k$.
\end{itemize}
The \emph{upper block index} $\bind_{\mm}^{U}(i)$ of a vertex $i$ is defined to
be the integer $r$ such that $i\in B_r$. Note that $\bind_{\mm}^{U}(i)$ is equal
to one more than the number of vertices incident to marked edges appearing
before $i$. The \emph{lower block index} $\bind_{\mm}^{L}(i)$ is defined
similarly. The \emph{block difference} $\bdiff_{\mm}(e)$ of an edge $e=(i,\pi(i))$ is defined by
$\bdiff_{\mm}(e)=\bind_{\mm}^{L}(\pi(i))-\bind_{\mm}^{U}(i)$. The definitions of
$\bwex(\mm)$ and $\wt(\mm)$ in Definition \ref{defn:stat_PMstar} are indeed
equivalent to those in Definition \ref{defn:stat_Mn} with $\bdiff_{\pi}$
replaced by $\bdiff_{\mm}$.
\end{remark}
\begin{figure}[H]
    \centering
    \resizebox{\linewidth}{!}{
    \begin{tikzpicture}
\pgfsetcornersarced{\pgfpoint{0.5mm}{0.5mm}}
    \tikzmath{\m = 0.15;\mh=0.3;\Mv=0.2;\Mh=0.45;\Mp=8.5;};
    
    \draw (1-\Mh,0-\Mv) rectangle (2+\Mh,1+\Mv);
    \draw (3-\Mh,0-\Mv) rectangle (5+\Mh,1+\Mv);
    \draw (6-\Mh,0-\Mv) rectangle (7+\Mh,1+\Mv);
    
    \filldraw[block] (1-\mh,1-\m) rectangle (2+\mh,1+\m);
    \filldraw[block] (3-\mh,1-\m) rectangle (4+\mh,1+\m);
    \filldraw[block] (5-\mh,1-\m) rectangle (7+\mh,1+\m);
    \filldraw[block] (1-\mh,0-\m) rectangle (1+\mh,0+\m);
    \filldraw[block] (2-\mh,0-\m) rectangle (5+\mh,0+\m);
    \filldraw[block] (6-\mh,0-\m) rectangle (7+\mh,0+\m);

    \draw[unmarked] (1,1) -- (2,0);
    \draw[unmarked] (3,1) -- (4,0);
    \draw[unmarked] (5,1) -- (3,0);
    \draw[unmarked] (6,1) -- (6,0);
    
    \draw[marked] (2,1) -- (1,0);
    \draw[marked] (4,1) -- (7,0);
    \draw[marked] (7,1) -- (5,0);


\foreach \x [count = \xi]  in {1,0,0,1,-1,0,-1}
{
    \fill (\xi,1) circle (1pt);
    \fill (\xi,0) circle (1pt);
    \node at (\xi,1.4) {$\x$};
    }

    \node[scale=0.7] at (2+0.2,1-0.05) {1};
    
    \node[scale=0.7] at (4+0.2,1-0.05) {2};
    \node[scale=0.7] at (7+0.2,1-0.05) {3};
    \node[scale=0.7] at (1+0.2,0-0.05) {1};
    \node[scale=0.7] at (5+0.2,0-0.05) {2};
    \node[scale=0.7] at (7+0.2,0-0.05) {3};
    \draw[thick] (8,0.5) -- (8.5,0.5);
    \draw[thick] (8.5,2) -- (8.5,-1);
    \draw[->,thick] (8.5,2) -- (9.5,2);
    \node at (9,2.3) {$\pi\setminus S$};
    \draw[->,thick] (8.5,-1) -- (9.5,-1);
    \node at (9,-0.7) {$\pi|_S$};
    \node[scale=1] at (2,-0.5) {$\bwex(\mm)=5$};
    \node[scale=1] at (5.5,-0.5) {$\wt(\mm)=2$};
    \node[scale=1] at (5.5,-1) {$\cross(\mm)=0$};
    
    \draw (1-\Mh+9.5,0-\Mv+1.5) rectangle (2+\Mh+9.5,1+\Mv+1.5);
    \draw (3-\Mh+9.5,0-\Mv+1.5) rectangle (5+\Mh+9.5,1+\Mv+1.5);
    \draw (6-\Mh+9.5,0-\Mv+1.5) rectangle (7+\Mh+9.5,1+\Mv+1.5);
    
    \filldraw[block] (1-\mh+9.5,1-\m+1.5) rectangle (2+\mh+9.5,1+\m+1.5);
    \filldraw[block] (3-\mh+9.5,1-\m+1.5) rectangle (4+\mh+9.5,1+\m+1.5);
    \filldraw[block] (5-\mh+9.5,1-\m+1.5) rectangle (7+\mh+9.5,1+\m+1.5);
    \filldraw[block] (1-\mh+9.5,0-\m+1.5) rectangle (1+\mh+9.5,0+\m+1.5);
    \filldraw[block] (2-\mh+9.5,0-\m+1.5) rectangle (5+\mh+9.5,0+\m+1.5);
    \filldraw[block] (6-\mh+9.5,0-\m+1.5) rectangle (7+\mh+9.5,0+\m+1.5);


    \draw[unmarked] (1+9.5,1+1.5) -- (2+9.5,0+1.5);
    \draw[unmarked] (3+9.5,1+1.5) -- (4+9.5,0+1.5);
    \draw[unmarked] (5+9.5,1+1.5) -- (3+9.5,0+1.5);
    \draw[unmarked] (6+9.5,1+1.5) -- (6+9.5,0+1.5);
    

\foreach \x [count = \xi]  in {1,,0,,-1,0,}
{
    \fill (\xi+9.5,1+1.5) circle (1pt);
    \fill (\xi+9.5,0+1.5) circle (1pt);
    \node at (\xi+9.5,1.4+1.5) {$\x$};
    }
    \node[scale=0.7] at (2+0.2+9.5,1-0.05+1.5) {1};
    \node[scale=0.7] at (4+0.2+9.5,1-0.05+1.5) {2};
    \node[scale=0.7] at (7+0.2+9.5,1-0.05+1.5) {3};
    \node[scale=0.7] at (1+0.2+9.5,0-0.05+1.5) {1};
    \node[scale=0.7] at (5+0.2+9.5,0-0.05+1.5) {2};
    \node[scale=0.7] at (7+0.2+9.5,0-0.05+1.5) {3};
    \node[scale=1] at (2.5+9.5,-0.5+1.5) {$\bwex(\pi\setminus S)=3$};
    \node[scale=1] at (5.5+9.5,-0.5+1.5) {$\bwt(\pi\setminus S)=1$};
    \node[scale=1] at (5.5+9.5,-1+1.5) {$\cross(\pi\setminus S)=1$};
    
\foreach \x in {1,2,3}
{
\filldraw[block] (\x-\mh+11.5,1-\m-1.5) rectangle (\x+\mh+11.5,1+\m-1.5);
\filldraw[block] (\x-\mh+11.5,0-\m-1.5) rectangle (\x+\mh+11.5,0+\m-1.5);
}


    \draw[marked] (1+11.5,1-1.5) -- (1+11.5,0-1.5);
    \draw[marked] (2+11.5,1-1.5) -- (3+11.5,0-1.5);
    \draw[marked] (3+11.5,1-1.5) -- (2+11.5,0-1.5);


\foreach \x [count = \xi]  in {0,1,-1}
{
    \fill (\xi+11.5,1-1.5) circle (1pt);
    \fill (\xi+11.5,0-1.5) circle (1pt);
    \node at (\xi+11.5,1.4-1.5) {$\x$};
    }
\foreach \x in {1,2,3}
{
    \node[scale=0.7] at (\x+0.2+11.5,1-0.05-1.5) {$\x$};
    \node[scale=0.7] at (\x+0.2+11.5,0-0.05-1.5) {$\x$};
    }
    \node[scale=1] at (2.5+9.5,-0.5-1.5) {$\bwex(\pi|_S)=2$};
    \node[scale=1] at (5.5+9.5,-0.5-1.5) {$\wt(\pi|_S)=1$};
    \node[scale=1] at (5.5+9.5,-1-1.5) {$\cross(\pi|_S)=1$};
\end{tikzpicture}
    }
    \caption{An example of a marked perfect matching $\mm$ in $\PM_{2,3,2}^{*}$ and its unmarked and marked portions.}
    \label{fig:stat_example}
\end{figure}
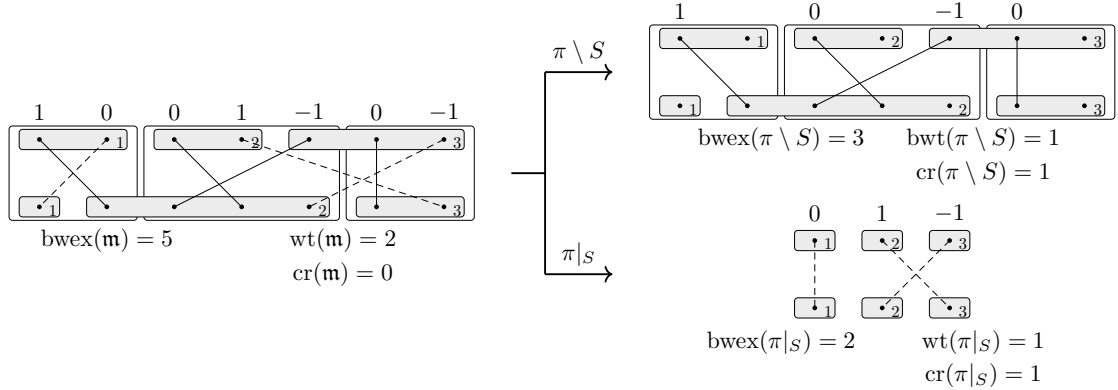
\begin{exam}
Figure \ref{fig:stat_example} shows a marked perfect matching $\mm$ in $\PM_{2,3,2}^{*}$ and block indices of its vertices. The block difference of each edge is indicated above its upper endpoint. The statistics $\bwex(\mm)=5$ and $\wt(\mm)=2$ can be computed directly by the notion of block difference in $\mm$, or summing the statistics defined on each $\pi\setminus S$ and $\pi|_S$. For the other statistics of $\mm$, we have $\ee(\mm)=4$ and $\cross(\mm)=0$.
\end{exam}
Under this construction, the linearization coefficient $C(n_1,\dots,n_k)$ is expressed in terms of marked perfect matchings by
\begin{equation}
\label{eq:cancnotfree}
C(n_1,\dots,n_k)=\sum_{\mm\in \PM^{*}_{n_1,\dots,n_k}}(-1)^{\ee(\mm)}y^{\bwex(\mm)}q^{\wt(\mm)+\cross(\mm)}.
\end{equation}
There are many cancellations in this summation. Our goal is to cancel all
negative terms by finding a sign-reversing involution on $\PM^{*}_{n_1,\dots,n_k}$.

Recall that $\DD(n_1,\dots,n_k)\subset S_N$ is the set of
$(n_1,\dots,n_k)$-derangements. The set $\DD(n_1,\dots,n_k)$ can be naturally
identified with the set of marked perfect matchings whose edges are all
inhomogeneous (necessarily marked). To be more precise, let $\sigma$ be a derangement in  $\DD(n_1,\dots,n_k)$. Then we will identify $\sigma$ with the marked perfect matching
$\mm=(\pi,E(\pi))\in \PM^*_{n_1,\dots,n_k}$, where $\pi\in\PM_{N}$ is given by $\pi(i)=\sigma(i)$ for all $i \in [N]$.
Under this identification one can easily check that 
\[
\wex(\sigma)=\wex(\pi) = \bwex(\mm),
\]
\[
\CR(\sigma) = \ov(\pi) = \wt(\pi)-\cross(\pi)=\wt(\mm)+\cross(\mm).
\]
By abuse of notation from now on we will write
\[
\DD(n_1,\dots,n_k) = \{(\pi,S)\in \PM^*_{n_1,\dots,n_k}: E^H(\pi)=\emptyset\}.
\]

Using the above discussion we can
rewrite Theorem~\ref{thm:Main} as follows.

\begin{thm}
\label{thm:Main2}
We have
\[
  C(n_1,\dots,n_k)=\sum_{\mm \in \DD(n_1,\dots,n_k)}
 y^{\bwex(\mm)}q^{\wt(\mm)+\cross(\mm)}.
\]
\end{thm}


\subsection{Construction of a sign-reversing involution}
\label{sec:involution_cons}
In order to prove Theorem \ref{thm:Main2}, we give a sign-reversing involution
$\Phi$ on $\PM^*_{n_1,\dots,n_k}$ that preserves the statistics $\bwex$ and
$\wt+\cross$. Indeed, $\Phi$ will be a map that marks or unmarks a single
homogeneous edge, or does not change anything. First we introduce some facts and
definitions that we need to describe the map $\Phi$.

For $\mm=(\pi,S)\in \PM^*_{n_1,\dots,n_k}$, let us observe a change in the block
difference of an edge $e_j$ while marking or unmarking a homogeneous edge $e_i$.
If we mark $e_i$ that was unmarked before, the upper (respectively, lower) index
$\bind_{\mm}^{U}(j)$ (respectively, $\bind_{\mm}^{L}(\pi(j))$) increases by $1$
if and only if $j>i$ (respectively, $\pi(j)>\pi(i)$). Therefore the block
difference $\bdiff_{\mm}(e_j)=\bind_{\mm}^{L}(\pi(j))-\bind_{\mm}^{U}(j)$
changes if and only if $e_j$ crosses $e_i$. More precisely, if $\mm=(\pi,S)$
with $e_i\not\in S$ turns into $\mm'=(\pi,S\cup\left\{e_i\right\})$, then we have
\begin{equation}
\label{eq:utm}
\bdiff_{\mm'}(e_j)=\left\{
\begin{array}{ll}
    \bdiff_{\mm}(e_j) &  \text{if $e_j=e_i$, or $e_j$ and $e_i$ do not cross each other,}\\
    \bdiff_{\mm}(e_j)+1 & \text{if $j<i$ and $\pi(j)>\pi(i)$,}\\
    \bdiff_{\mm}(e_j)-1 & \text{if $j>i$ and $\pi(j)<\pi(i)$.}
\end{array}\right. 
\end{equation}
Conversely, if we unmark a marked edge $e_i\in E^H(\pi)$ so that $\mm=(\pi,S)$
turns into $\mm'=(\pi,S\setminus\left\{e_i\right\})$, then we have
\begin{equation}
\label{eq:mtu}
\bdiff_{\mm'}(e_j)=\left\{
\begin{array}{ll}
    \bdiff_{\mm}(e_j) &  \text{if $e_j=e_i$, or $e_j$ and $e_i$ do not cross each other,}\\
    \bdiff_{\mm}(e_j)-1 & \text{if $j<i$ and $\pi(j)>\pi(i)$,}\\
    \bdiff_{\mm}(e_j)+1 & \text{if $j>i$ and $\pi(j)<\pi(i)$.}
\end{array}\right. 
\end{equation}
From now on, let us adopt an expression \emph{$e_j$ crosses $e_i$ from the left}, or equivalently \emph{$e_i$ crosses $e_j$ from the right} for the relation $j<i$ and $\pi(j)>\pi(i)$. With this observation, we define the \emph{convertibility} of a homogeneous edge, which is a key ingredient of the map $\Phi$.

\begin{defn}
\label{defn:convertible}
Let $\mm=(\pi,S) \in \PM^{*}_{n_1,\dots,n_k}$. An edge $e\in E^H(\pi)$ is said to be \emph{convertible} (in $\mm$) if it satisfies the following conditions.
\begin{enumerate}
    \item If $e$ is unmarked, i.e., $e\notin S$, then for every edge $e'$ that crosses $e$, either
    \begin{itemize}
        \item $e'$ crosses $e$ from the left and $\bdiff_{\mm}(e')\geq0$, or  
        \item $e'$ crosses $e$ from the right and $\bdiff_{\mm}(e')\leq-1$.
    \end{itemize}
    \item If $e$ is marked, i.e., $e\in S$, then for every edge $e'$ that crosses $e$, either
    \begin{itemize}
        \item $e'$ crosses $e$ from the left and $\bdiff_{\mm}(e')>0$, or  
        \item $e'$ crosses $e$ from the right and $\bdiff_{\mm}(e')<-1$.
    \end{itemize}
\end{enumerate}
\end{defn}
Note that if an edge $e\in E^H(\pi)$ is convertible, then the status of other
edges being block weak excedances does not change under the map
$\mm=(\pi,S)\mapsto \mm'=(\pi,S\triangle\{e\})$, where $X\triangle Y$ denotes
the symmetric difference $(X\cup Y)\setminus(X\cap Y)$. In particular, marking
or unmarking a convertible edge preserves the statistic $\bwex$. Note also that
an edge $e$ is convertible in $\mm=(\pi,S)$ if and only if it is convertible in
$\mm'=(\pi,S\triangle\{e\})$.
\begin{remark}
  Suppose that $e'=(i,\pi(i))$ is an inhomogeneous edge of
  $\mm=(\pi,S)\in\M_{n_1,\dots,n_k}$. Then $n_1+\dots+n_{r-1}+1 \le i\le
  n_1+\dots+n_r$ and $n_1+\dots+n_{s-1}+1 \le \pi(i)\le n_1+\dots+n_s$ for some
  $r\ne s$. It is easy to check that the block difference $\bdiff_{\mm}(e')$ is
  nonzero, and its sign is determined by $r$ and $s$. Thus, marking or unmarking
  a homogeneous edge $e\in E^H(\mm)$ does not change the status of whether $e'$ is a block
  weak excedance or not. Therefore, it is sufficient to consider the changes of block
  differences of homogeneous edges when we toggle $e$.
\end{remark}
We are now ready to define the involution $\Phi$.
\begin{defn}\label{Involution}
[The involution $\Phi$]
For $\mm=(\pi,S) \in \PM^{*}_{n_1,\dots,n_k}$, we define $\Phi(\mm)$ as follows.
\begin{spacing}{1.5}

\end{spacing}
\noindent\textbf{Case 0} If $\mm$ has no homogeneous edges, then define $\Phi(\mm)=\mm$. In other words, $\Phi$ is the identity map on $\DD(n_1,\dots,n_k)$.
\begin{spacing}{1.5}

\end{spacing}
\noindent\textbf{Case 1} Suppose $\mm$ has homogeneous edges and $\bdiff_{\mm}(e)\geq0$ for
  all $e \in E^H(\pi)$. Define $\Phi(\mm)=(\pi,S\triangle\left\{e_i\right\})$,
  where $i$ is the integer satisfying $\pi(i)=\min\{\pi(j):$ $e_j\in E^H(\pi)\}$. In other words, we mark or unmark the homogeneous edge
  whose lower endpoint is the leftmost one among the homogeneous edges.
\begin{spacing}{1.5}

\end{spacing}
\noindent\textbf{Case 2} Suppose $\mm$ has homogeneous edges and $\bdiff_{\mm}(e)<0$ for some $e \in E^H(\pi)$. Let $i=\min\left\{j:e_j\in E^H(\pi), \bdiff_{\mm}(e_j)<0\right\}$. Depending on the convertibility of the edge $e_i$, we consider two subcases.
\begin{description}
\item[Subcase 2-(a)] If $e_i$ is convertible, then define $\Phi(\mm)=(\pi,S\triangle\left\{e_{i}\right\})$.
\item [Subcase 2-(b)] If $e_i$ is not convertible, then define
  $\Phi(\mm)=(\pi,S\triangle\left\{e_{i'}\right\})$, where
\[
i'=\max\left\{j<i:e_j\in E^H(\pi), \bdiff_{\mm}(e_j)=0, e_j\mbox{ crosses }e_i\right\}.
\]
\end{description}
\end{defn}
\begin{exam}
  The applications of the map $\Phi$ in Cases 1, 2-(a) and 2-(b) are illustrated
  in Figures \ref{fig:Case1},\ref{fig:Case2a} and \ref{fig:Case2b},
  respectively. Marked edges are represented in dashed lines, and inhomogeneous
  edges are colored in gray. The block differences of homogeneous edges are
  indicated by the numbers above their upper endpoints. The edge chosen by
  $\Phi$ is the thick (dashed) edge.
\end{exam}
\begin{figure}[H]
    \centering
    \begin{tikzpicture}
\pgfsetcornersarced{\pgfpoint{0.5mm}{0.5mm}}
    \tikzmath{\m = 0.15;\mh=0.3;\Mv=0.2;\Mh=0.45;};
    
    \draw (4-\Mh,0-\Mv) rectangle (5+\Mh,1+\Mv);
    \draw (1-\Mh,0-\Mv) rectangle (3+\Mh,1+\Mv);
    \filldraw[block] (1-\mh,1-\m) rectangle (2+\mh,1+\m);
    \filldraw[block] (3-\mh,1-\m) rectangle (4+\mh,1+\m);
    \filldraw[block] (5-\mh,1-\m) rectangle (5+\mh,1+\m);
    \filldraw[block] (1-\mh,0-\m) rectangle (1+\mh,0+\m);
    \filldraw[block] (2-\mh,0-\m) rectangle (4+\mh,0+\m);
    \filldraw[block] (5-\mh,0-\m) rectangle (5+\mh,0+\m);
    \draw[unmarked] (1,1) -- (3,0);
    \draw[marked0] (2,1) -- (4,0);
    \draw[unmarked,ultra thick] (3,1) -- (2,0);
    \draw[marked0] (4,1) -- (1,0);
    \draw[unmarked] (5,1) -- (5,0);
    \node[scale=0.7] at (2+0.2,1-0.05) {1};
    \node[scale=0.7] at (4+0.2,1-0.05) {2};
    \node[scale=0.7] at (5+0.2,1-0.05) {3};
    \node[scale=0.7] at (1+0.2,0-0.05) {1};
    \node[scale=0.7] at (4+0.2,0-0.05) {2};
    \node[scale=0.7] at (5+0.2,0-0.05) {3};
\foreach \x [count = \xi] in {1, ,0, ,0}
{
    \fill (\xi,1) circle (1pt);
    \fill (\xi,0) circle (1pt);
    \node at (\xi,1.4) {$\x$};
}
    \draw[<->] (3,-0.3) edge (3,-1.3);
    \node at (3.3,-0.8) {$\Phi$};
    
    \draw (4-\Mh,0-\Mv-3) rectangle (5+\Mh,1+\Mv-3);
    \draw (1-\Mh,0-\Mv-3) rectangle (3+\Mh,1+\Mv-3);
    \filldraw[block] (1-\mh,1-\m-3) rectangle (2+\mh,1+\m-3);
    \filldraw[block] (3-\mh,1-\m-3) rectangle (3+\mh,1+\m-3);
    \filldraw[block] (4-\mh,1-\m-3) rectangle (4+\mh,1+\m-3);
    \filldraw[block] (5-\mh,1-\m-3) rectangle (5+\mh,1+\m-3);
    \filldraw[block] (1-\mh,0-\m-3) rectangle (1+\mh,0+\m-3);
    \filldraw[block] (2-\mh,0-\m-3) rectangle (2+\mh,0+\m-3);
    \filldraw[block] (3-\mh,0-\m-3) rectangle (4+\mh,0+\m-3);
    \filldraw[block] (5-\mh,0-\m-3) rectangle (5+\mh,0+\m-3);
    \draw[unmarked] (1,1-3) -- (3,0-3);
    \draw[marked0] (2,1-3) -- (4,0-3);
    \draw[marked,ultra thick] (3,1-3) -- (2,0-3);
    \draw[marked0] (4,1-3) -- (1,0-3);
    \draw[unmarked] (5,1-3) -- (5,0-3);
    \node[scale=0.7] at (2+0.2,1-0.05-3) {1};
    \node[scale=0.7] at (3+0.2,1-0.05-3) {2};
    \node[scale=0.7] at (4+0.2,1-0.05-3) {3};
    \node[scale=0.7] at (5+0.2,1-0.05-3) {4};
    \node[scale=0.7] at (1+0.2,0-0.05-3) {1};
    \node[scale=0.7] at (2+0.2,0-0.05-3) {2};
    \node[scale=0.7] at (4+0.2,0-0.05-3) {3};
    \node[scale=0.7] at (5+0.2,0-0.05-3) {4};
\foreach \x [count = \xi] in {2, ,0, ,0}
{
    \fill (\xi,1-3) circle (1pt);
    \fill (\xi,0-3) circle (1pt);
    \node at (\xi,1.4-3) {$\x$};
}
\end{tikzpicture}
    \caption{An example of the map $\Phi$ in Case 1, which toggles the edge $e_3=(3,\overline{2})$.}
    \label{fig:Case1}
\end{figure}
\begin{figure}[H]
    \centering
     \begin{tikzpicture}
\pgfsetcornersarced{\pgfpoint{0.5mm}{0.5mm}}
    \tikzmath{\m = 0.15;\mh=0.3;\Mv=0.2;\Mh=0.45;};
    
    \draw (4-\Mh,0-\Mv) rectangle (7+\Mh,1+\Mv);
    \draw (1-\Mh,0-\Mv) rectangle (3+\Mh,1+\Mv);
    \filldraw[block] (1-\mh,1-\m) rectangle (3+\mh,1+\m);
    \filldraw[block] (4-\mh,1-\m) rectangle (5+\mh,1+\m);
    \filldraw[block] (6-\mh,1-\m) rectangle (6+\mh,1+\m);
    \filldraw[block] (7-\mh,1-\m) rectangle (7+\mh,1+\m);
    \filldraw[block] (1-\mh,0-\m) rectangle (3+\mh,0+\m);
    \filldraw[block] (4-\mh,0-\m) rectangle (4+\mh,0+\m);
    \filldraw[block] (5-\mh,0-\m) rectangle (6+\mh,0+\m);
    \filldraw[block] (7-\mh,0-\m) rectangle (7+\mh,0+\m);
    \draw[unmarked] (1,1) -- (2,0);
    \draw[unmarked] (2,1) -- (1,0);
    \draw[marked0] (3,1) -- (6,0);
    \draw[unmarked] (4,1) -- (5,0);
    \draw[marked] (5,1) -- (7,0);
    \draw[marked, ultra thick] (6,1) -- (4,0);
    \draw[marked0] (7,1) -- (3,0);
    \node[scale=0.7] at (3+0.2,1-0.05) {1};
    \node[scale=0.7] at (5+0.2,1-0.05) {2};
    \node[scale=0.7] at (6+0.2,1-0.05) {3};
    \node[scale=0.7] at (7+0.2,1-0.05) {4};
    \node[scale=0.7] at (3+0.2,0-0.05) {1};
    \node[scale=0.7] at (4+0.2,0-0.05) {2};
    \node[scale=0.7] at (6+0.2,0-0.05) {3};
    \node[scale=0.7] at (7+0.2,0-0.05) {4};
\foreach \x [count = \xi] in {0,0, ,1,2,-1, }
{
    \fill (\xi,1) circle (1pt);
    \fill (\xi,0) circle (1pt);
    \node at (\xi,1.4) {$\x$};
}
    \draw[<->] (3.75,-0.3) edge (3.75,-1.3);
    \node at (4.05,-0.8) {$\Phi$};
    
    \draw (4-\Mh,0-\Mv-3) rectangle (7+\Mh,1+\Mv-3);
    \draw (1-\Mh,0-\Mv-3) rectangle (3+\Mh,1+\Mv-3);
    \filldraw[block] (1-\mh,1-\m-3) rectangle (3+\mh,1+\m-3);
    \filldraw[block] (4-\mh,1-\m-3) rectangle (5+\mh,1+\m-3);
    \filldraw[block] (6-\mh,1-\m-3) rectangle (7+\mh,1+\m-3);
    \filldraw[block] (1-\mh,0-\m-3) rectangle (3+\mh,0+\m-3);
    \filldraw[block] (4-\mh,0-\m-3) rectangle (6+\mh,0+\m-3);
    \filldraw[block] (7-\mh,0-\m-3) rectangle (7+\mh,0+\m-3);
    \draw[unmarked] (1,1-3) -- (2,0-3);
    \draw[unmarked] (2,1-3) -- (1,0-3);
    \draw[marked0] (3,1-3) -- (6,0-3);
    \draw[unmarked] (4,1-3) -- (5,0-3);
    \draw[marked] (5,1-3) -- (7,0-3);
    \draw[unmarked,ultra thick] (6,1-3) -- (4,0-3);
    \draw[marked0] (7,1-3) -- (3,0-3);
    \node[scale=0.7] at (3+0.2,1-0.05-3) {1};
    \node[scale=0.7] at (5+0.2,1-0.05-3) {2};
    \node[scale=0.7] at (7+0.2,1-0.05-3) {3};
    \node[scale=0.7] at (3+0.2,0-0.05-3) {1};
    \node[scale=0.7] at (6+0.2,0-0.05-3) {2};
    \node[scale=0.7] at (7+0.2,0-0.05-3) {3};
\foreach \x [count = \xi] in {0,0, ,0,1,-1, }
{
    \fill (\xi,1-3) circle (1pt);
    \fill (\xi,0-3) circle (1pt);
    \node at (\xi,1.4-3) {$\x$};
}
\end{tikzpicture}
    \caption{An example of the map $\Phi$ in Subcase 2-(a), which toggles the edge $e_6=(6,\overline{4})$.}
    \label{fig:Case2a}
\end{figure}
\begin{figure}[H]
    \centering
    \begin{tikzpicture}
\pgfsetcornersarced{\pgfpoint{0.5mm}{0.5mm}}
    \tikzmath{\m = 0.15;\mh=0.3;\Mv=0.2;\Mh=0.45;};
    
    \draw (3-\Mh,0-\Mv) rectangle (7+\Mh,1+\Mv);
    \draw (1-\Mh,0-\Mv) rectangle (2+\Mh,1+\Mv);
    \filldraw[block] (1-\mh,1-\m) rectangle (2+\mh,1+\m);
    \filldraw[block] (3-\mh,1-\m) rectangle (3+\mh,1+\m);
    \filldraw[block] (4-\mh,1-\m) rectangle (4+\mh,1+\m);
    \filldraw[block] (5-\mh,1-\m) rectangle (6+\mh,1+\m);
    \filldraw[block] (7-\mh,1-\m) rectangle (7+\mh,1+\m);
    \filldraw[block] (1-\mh,0-\m) rectangle (1+\mh,0+\m);
    \filldraw[block] (2-\mh,0-\m) rectangle (3+\mh,0+\m);
    \filldraw[block] (4-\mh,0-\m) rectangle (4+\mh,0+\m);
    \filldraw[block] (5-\mh,0-\m) rectangle (7+\mh,0+\m);
    \draw[unmarked] (1,1) -- (2,0);
    \draw[marked0] (2,1) -- (3,0);
    \draw[marked0] (3,1) -- (1,0);
    \draw[marked] (4,1) -- (7,0);
    \draw[unmarked,ultra thick] (5,1) -- (5,0);
    \draw[marked] (6,1) -- (4,0);
    \draw[unmarked] (7,1) -- (6,0);
    \node[scale=0.7] at (2+0.2,1-0.05) {1};
    \node[scale=0.7] at (3+0.2,1-0.05) {2};
    \node[scale=0.7] at (4+0.2,1-0.05) {3};
    \node[scale=0.7] at (6+0.2,1-0.05) {4};
    \node[scale=0.7] at (7+0.2,1-0.05) {5};
    \node[scale=0.7] at (1+0.2,0-0.05) {1};
    \node[scale=0.7] at (3+0.2,0-0.05) {2};
    \node[scale=0.7] at (4+0.2,0-0.05) {3};
    \node[scale=0.7] at (7+0.2,0-0.05) {4};
\foreach \x [count = \xi] in {1, , ,1,0,-1,-1 }
{
    \fill (\xi,1) circle (1pt);
    \fill (\xi,0) circle (1pt);
    \node at (\xi,1.4) {$\x$};
}
    \draw[<->] (3.75,-0.3) edge (3.75,-1.3);
    \node at (4.05,-0.8) {$\Phi$};
    
    \draw (3-\Mh,0-\Mv-3) rectangle (7+\Mh,1+\Mv-3);
    \draw (1-\Mh,0-\Mv-3) rectangle (2+\Mh,1+\Mv-3);
    \filldraw[block] (1-\mh,1-\m-3) rectangle (2+\mh,1+\m-3);
    \filldraw[block] (3-\mh,1-\m-3) rectangle (3+\mh,1+\m-3);
    \filldraw[block] (4-\mh,1-\m-3) rectangle (4+\mh,1+\m-3);
    \filldraw[block] (5-\mh,1-\m-3) rectangle (5+\mh,1+\m-3);
    \filldraw[block] (6-\mh,1-\m-3) rectangle (6+\mh,1+\m-3);
    \filldraw[block] (7-\mh,1-\m-3) rectangle (7+\mh,1+\m-3);
    \filldraw[block] (1-\mh,0-\m-3) rectangle (1+\mh,0+\m-3);
    \filldraw[block] (2-\mh,0-\m-3) rectangle (3+\mh,0+\m-3);
    \filldraw[block] (4-\mh,0-\m-3) rectangle (4+\mh,0+\m-3);
    \filldraw[block] (5-\mh,0-\m-3) rectangle (5+\mh,0+\m-3);
    \filldraw[block] (6-\mh,0-\m-3) rectangle (7+\mh,0+\m-3);
    \draw[unmarked] (1,1-3) -- (2,0-3);
    \draw[marked0] (2,1-3) -- (3,0-3);
    \draw[marked0] (3,1-3) -- (1,0-3);
    \draw[marked] (4,1-3) -- (7,0-3);
    \draw[marked,ultra thick] (5,1-3) -- (5,0-3);
    \draw[marked] (6,1-3) -- (4,0-3);
    \draw[unmarked] (7,1-3) -- (6,0-3);
    \node[scale=0.7] at (2+0.2,1-0.05-3) {1};
    \node[scale=0.7] at (3+0.2,1-0.05-3) {2};
    \node[scale=0.7] at (4+0.2,1-0.05-3) {3};
    \node[scale=0.7] at (5+0.2,1-0.05-3) {4};
    \node[scale=0.7] at (6+0.2,1-0.05-3) {5};
    \node[scale=0.7] at (7+0.2,1-0.05-3) {6};
    \node[scale=0.7] at (1+0.2,0-0.05-3) {1};
    \node[scale=0.7] at (3+0.2,0-0.05-3) {2};
    \node[scale=0.7] at (4+0.2,0-0.05-3) {3};
    \node[scale=0.7] at (5+0.2,0-0.05-3) {4};
    \node[scale=0.7] at (7+0.2,0-0.05-3) {5};
\foreach \x [count = \xi] in {1,,,2,0,-2,-1 }
{
    \fill (\xi,1-3) circle (1pt);
    \fill (\xi,0-3) circle (1pt);
    \node at (\xi,1.4-3) {$\x$};
}
\end{tikzpicture}
    \caption{An example of the map $\Phi$ in Subcase 2-(b), which toggles the edge $e_5=(5,\overline{5})$.}
    \label{fig:Case2b}
\end{figure}

For the well-definedness of $\Phi$, the only part that is not clear is the
existence of the number $i'$ in Subcase 2-(b), or equivalently, 
\[
\left\{j<i : e_j\in E^H(\pi), \bdiff_{\mm}(e_j)=0, e_j\mbox{ crosses }e_i\right\}\neq\emptyset.
\]
We will prove this in Lemma \ref{lem:iprime_exists}.

Note that except for Case 0, $\Phi$ toggles only one edge's marking status. Hence $\Phi$ is sign-reversing. In the following section, we will prove that $\Phi$ is indeed a well-defined involution that preserves the statistics $\bwex$ and $\wt+\cross$.

\section{Proof of Theorem~\ref{thm:Main2}}\label{sec:proofinvolution}
We start with a simple fact which will be used frequently throughout this section.
\begin{prop}
\label{prop:cr_bdf}
Let $e_i$ and $e_j$ be edges in $\mm=(\pi,S)$ such that $e_i$ crosses $e_j$ from the left, or equivalently, $e_j$ crosses $e_i$ from the right. Then we have
$$
\bdiff_{\mm}(e_i)\geq\bdiff_{\mm}(e_j).
$$
Moreover, the inequality is strict if $e_i\in S$ or $e_j\in S$.
\end{prop}
\begin{proof}
By the assumption, we have $i<j$ and
$\pi(i)>\pi(j)$ and the first statement follows from the relations
\begin{align}
\bind_{\mm}^{U}(i)&\leq\bind_{\mm}^{U}(j)\label{eqn:leqU} \quad\mbox{and} \\
\bind_{\mm}^{L}(\pi(i))&\geq\bind_{\mm}^{L}(\pi(j))\label{eqn:leqL}.
\end{align}
Since the inequality \eqref{eqn:leqU} (respectively, \eqref{eqn:leqL}) holds
strictly if $e_i\in S$ (respectively, $e_j\in S$), we obtain the second statement.
\end{proof}
Before proving Theorem~\ref{thm:Main2}, we verify the well-definedness of $\Phi$ in the
following two lemmas. Recall that $\DD(n_1,\dots,n_k)$ is identified with the
set of marked perfect matchings in $\PM^{*}_{n_1,\dots,n_k}$ such that all edges
are inhomogeneous. 
\begin{lem}
\label{lem:i_is_marked}
Let $\mm=(\pi,S) \in \PM^{*}_{n_1,\dots,n_k}\setminus \DD(n_1,\dots,n_k)$. Suppose $e=e_i\in E^H(\pi)$ is not convertible, where $$i=\min\left\{j:e_j\in E^H(\pi),\, \bdiff_{\mm}(e_j)<0\right\}.$$
Then $e$ is a marked edge, i.e., $e\in S$.
\end{lem}
\begin{proof}
  Suppose that $e$ is not marked, i.e., $e \notin S$. By the assumption that
  $e$ is not convertible, there are two possibilities:
\begin{itemize}
\item There is an edge $e'\in E^H(\pi)$ such that $e'$ crosses $e$ from the
  left and $\bdiff_{\mm}(e')<0$, or
\item There is an edge $e'\in E^H(\pi)$ such that $e'$ crosses $e$ from the
  right and $\bdiff_{\mm}(e')>-1$.
\end{itemize}
By the minimality of $i$, the first case cannot occur. For the second
case, since $e'$ crosses $e$ from the right and $\bdiff_{\mm}(e)<0$, we have
$\bdiff_{\mm}(e')\leq\bdiff_{\mm}(e)<0$, which contradicts the fact that
$\bdiff_{\mm}(e')>-1$. Therefore, $e$ is a marked edge.
\end{proof}
\begin{lem}
\label{lem:iprime_exists}
Under the same assumptions in Lemma \ref{lem:i_is_marked}, the set
$$
\left\{j<i:e_j\in E^H(\pi), \bdiff_{\mm}(e_j)=0, e_j\mbox{ crosses }e_i\right\}
$$
is not empty.
\end{lem}
\begin{proof}
  By Lemma \ref{lem:i_is_marked}, $e=e_i$ is marked. By the assumption that $e$
  is not convertible, we have two possible cases:
\begin{itemize}
\item There is an edge $e'\in E^H(\pi)$ such that $e'$ crosses $e$ from the
  left and $\bdiff_{\mm}(e')\le 0$, or
\item There is an edge $e'\in E^H(\pi)$ such that $e'$ crosses $e$ from the
  right and $\bdiff_{\mm}(e')\ge -1$.
\end{itemize}
Suppose an edge $e'\in E^H(\pi)$ crosses $e$ from the right. Then we have
$\bdiff_{\mm}(e')<\bdiff_{\mm}(e)<0$, where the first inequality follows from
the fact that $e$ is marked. Thus the latter case cannot happen. Therefore there
exists an edge $e'$ corresponding to the first case. By the minimality of $i$,
we have $\bdiff_{\mm}(e')=0$ and, therefore, the given set is not empty.
\end{proof}

We now give a proof of Theorem~\ref{thm:Main2} by a sequence of lemmas. The
first objective is to prove that the edge chosen by $\Phi$ to be toggled is
convertible.
\begin{lem}
\label{lem:convertible}
For $\mm=(\pi,S)\in \PM_{n_1,\dots,n_k}^{*}\setminus \DD(n_1,\dots,n_k)$, the edge in $\mm$ toggled by the map $\Phi$ is convertible in $\mm$, i.e., if $\Phi(\mm)=(\pi,S\triangle\left\{e\right\})$, then the edge $e$ is convertible in $\mm$.
\end{lem}
\begin{proof}
We consider each case in Definition~\ref{Involution} except for Case 0,
which does not occur since $\mm\not\in \DD(n_1,\dots,n_k)$.
\begin{spacing}{1.5}

\end{spacing}
\noindent\textbf{Case 1} Suppose $\mm$ has homogeneous edges and $\bdiff_{\mm}(e)\geq0$ for all $e \in E^H(\pi)$. Then $\Phi(\mm)=(\pi,S\triangle\left\{e_i\right\})$ for the integer $i$ satisfying $\pi(i)=\min\left\{\pi(j):e_j\in E^H(\pi)\right\}$.  Suppose an edge $e\in E^H(\pi)$ crosses $e_i$. By the minimality of $\pi(i)$, $e$ must cross $e_i$ from the left. By Proposition~\ref{prop:cr_bdf}, we have  $\bdiff_{\mm}(e_i)\leq\bdiff_{\mm}(e)$ if $e_i$ is unmarked and  $\bdiff_{\mm}(e_i)<\bdiff_{\mm}(e)$ if $e_i$ is marked. Since  $\bdiff_{\mm}(e_i)\ge0$ we conclude that $e_i$ is convertible. 
\begin{spacing}{1.5}

\end{spacing}
\noindent\textbf{Case 2} Suppose $\mm$ has homogeneous edges and $\bdiff_{\mm}(e)<0$ for
  some $e \in E^H(\pi)$. Let 
$$
i=\min\{j:e_j\in E^H(\pi),\bdiff_{\mm}(e_j)<0\}.
$$
There are two subcases.
\begin{description}[]
\item[Subcase 2-(a)] If $e_i$ is convertible, then
  $\Phi(\mm)=(\pi,S\triangle\left\{e_i\right\})$. In this case $e_i$ is
  convertible by the assumption.
\item[Subcase 2-(b)] If $e_i$ is not convertible, then
  $\Phi(\mm)=(\pi,S\triangle\left\{e_{i'}\right\})$, where
    $$
    i'=\max\left\{j<i:e_j\in E^H(\pi), \bdiff_{\mm}(e_j)=0, e_j \text{ crosses } e_i\right\}.
    $$
    Let $e=e_i$ and $e'=e_{i'}$. To check the convertibility of $e'$, let
    $e''=e_{i''}\in E^H(\pi)$ be an edge that crosses $e'$.
    
    If $e''$ crosses $e'$ from the left, then we have
    $\bdiff_{\mm}(e'')\geq\bdiff_{\mm}(e')=0$, where the inequality is strict if
    $e'$ is marked. Now suppose that $e''$ crosses $e'$ from the right. Then we have
    $\bdiff_{\mm}(e'')\leq\bdiff_{\mm}(e')=0$, where the inequality is strict if
    $e'$ is marked. For the convertibility of $e'$ we have to show that the
    inequality $\bdiff_{\mm}(e'')\leq -1$ (respectively, $\bdiff_{\mm}(e'')<-1$)
    holds if $e'\notin S$ (respectively, $e'\in S$). We consider the two cases
    $e'\not\in S$ and $e\in S$ as follows.
    \begin{description}
    \item [$\bm{e'\notin S}$] Suppose $\bdiff_{\mm}(e'')=\bdiff_{\mm}(e')=0$. By
      the argument in the proof of Proposition \ref{prop:cr_bdf}, it follows
      that $\bind_{\mm}^{U}(i')=\bind_{\mm}^{U}(i'')$ and
      $\bind_{\mm}^{L}(\pi(i'))=\bind_{\mm}^{L}(\pi(i''))$. On the other hand,
      the edge $e$ is marked and it crosses $e'$ from the right, by the choice
      of $e'$ and Lemma \ref{lem:i_is_marked}. From these facts, we must have
      $i'<i''<i$ and $\pi(i')>\pi(i'')>\pi(i)$. Thus, $e''$ crosses $e$ from the
      left and $\bdiff_{\mm}(e'')=0$, but this contradicts the maximality of $i'$.
      Therefore we have $\bdiff_{\mm}(e'')<\bdiff_{\mm}(e')$ and
      $\bdiff_{\mm}(e'')\leq -1$.
    \item [$\bm{e'\in S}$] Suppose $\bdiff_{\mm}(e'')=\bdiff_{\mm}(e')-1=-1$. By
      the same argument in the case $e'\notin S$, we have
      $\bind_{\mm}^{U}(i')=\bind_{\mm}^{U}(i'')-1$,
      $\bind_{\mm}^{L}(\pi(i'))=\bind_{\mm}^{L}(\pi(i''))$, $i'<i''\leq i$
      and $\pi(i')>\pi(i'')\geq\pi(i)$. By the minimality of $i$, we must have
      $i''=i$, but then
      $\bind_{\mm}^{L}(\pi(i'))>\bind_{\mm}^{L}(\pi(i))=\bind_{\mm}^{L}(\pi(i''))$
      also follows, which contradicts 
      $\bind_{\mm}^{L}(\pi(i'))=\bind_{\mm}^{L}(\pi(i''))$. Hence we have
      $\bdiff_{\mm}(e'')<\bdiff_{\mm}(e')-1=-1$.
\end{description}
\end{description}
\end{proof}

\begin{lem}
\label{lem:involution}
The map $\Phi$ is an involution, i.e., $\Phi^2=\ID$.
\end{lem}
\begin{proof}
  Let $\mm=(\pi,S) \in \PM_{n_1,\dots,n_k}^{*}$ be mapped to $\mm' $ by $\Phi$.
  We will show that $\Phi(\mm')=\mm$ in each case of Definition~\ref{Involution},
  where Case 0 is trivial.
\begin{spacing}{1.5}

\end{spacing}
\noindent\textbf{Case 1} Suppose $\mm$ has homogeneous edges and $\bdiff_{\mm}(e)\geq0$ for
  all $e \in E^H(\pi)$. Then $\mm'=\Phi(\mm)=(\pi,S\triangle\left\{e_i\right\})$
  for the integer $i$ satisfying $\pi(i)=\min\left\{\pi(j):e_j\in
    E^H(\pi)\right\}$. Since $e_i$ is convertible by
  Lemma~\ref{lem:convertible}, we also have $\bdiff_{\mm'}(e)\ge 0$ for all $e \in
  E^H(\pi)$. That is, when we apply $\Phi$ to $\mm'$ we are still in Case 1.
  Since the edge $e_i$ toggled by $\Phi$ depends only on $\pi$, the map $\Phi$
  toggles the same edge $e_i$ in $\mm'$ and we have $\Phi(\mm')=\mm$.
\begin{spacing}{1.5}

\end{spacing}
\noindent\textbf{Case 2} Suppose $\mm$ has homogeneous edges and $\bdiff_{\mm}(e)<0$ for
  some $e \in E^H(\pi)$. Let 
$$
i=\min\{j:e_j\in E^H(\pi),\bdiff_{\mm}(e_j)<0\}.
$$
Note that in this case, the set of edges having negative block difference is
invariant under $\Phi$ by the convertibility of the edge toggled by $\Phi$.
Hence when we apply $\Phi$ to $\mm'$ we are still in Case 2 and the same index
$i$ satisfies 
$$
i=\min\{j:e_j\in E^H(\pi),\bdiff_{\mm'}(e_j)<0\}.
$$
Now we consider the following two subcases.

\begin{description}[listparindent=\parindent]
\item [Subcase 2-(a)] If $e_i$ is convertible, then
  $\Phi(\mm)=(\pi,S\triangle\left\{e_i\right\})$. Since $e_i$ is also convertible in
  $\mm'$, we have $\Phi(\mm')=\mm$.
\item [Subcase 2-(b)] If $e_i$ is not convertible, then
  $\Phi(\mm)=(\pi,S\triangle\left\{e_{i'}\right\})$, where
    $$
    i'=\max\left\{j<i:e_j\in E^H(\pi), \bdiff_{\mm}(e_j)=0, e_j \text{ crosses } e_i\right\}.
    $$
Let $e=e_i$ and $e'=e_{i'}$. To show $\Phi(\mm')=\mm$, it suffices to show the
following two properties:
\begin{itemize}
    \item $e$ is not convertible in $\mm'$, i.e., when we apply $\Phi$ to $\mm'$
      we are still in Subcase 2-(b).
    \item The map $\Phi$ toggles the same edge $e'$ in $\mm'$, i.e.,
      \begin{equation}
        \label{eq:2}
        i'=\max\left\{j<i : e_j\in E^H(\pi),\bdiff_{\mm'}(e_j)=0,e_j
          \text{ crosses } e_i \right\}.
      \end{equation}
\end{itemize}

For the first part, recall that toggling $e'$ preserves the block difference of
itself. That is, $\bdiff_{\mm}(e')=\bdiff_{\mm'}(e')=0$. Therefore $e$ is still
not convertible in $\mm'$ due to the presence of $e'$.

For the second part, we will prove \eqref{eq:2}. Assume that there is an edge
$e''=e_{i''}\in E^H(\pi)$ such that $i'<i''<i$, $\bdiff_{\mm'}(e'')=0$ and $e''$
crosses $e$. In $\mm$, we must have $\bdiff_{\mm}(e'')>0$ by the choice of the
indices $i$ and $i'$ in $\mm$. Since $\mm$ and $\mm'$ only differ by a mark on
$e'$, the only possible case is that $\bdiff_{\mm}(e'')=1$, $e'$ crosses $e''$
from left and $e'$ is marked in $\mm$. Hence we must have
$\bdiff_{\mm}(e')>\bdiff_{\mm}(e'')=1$, but this contradicts the assumption
$\bdiff_{\mm}(e')=0$. Therefore, there is no such edge $e''$ and \eqref{eq:2}
holds.
\end{description}
\end{proof}
\begin{lem}
\label{lem:bwex}
The map $\Phi$ preserves the block weak excedances, i.e., $$\bwex(\mm)=\bwex(\Phi(\mm)).$$
\end{lem}
\begin{proof}
It is an immediate consequence of Lemma \ref{lem:convertible}.
\end{proof}

\begin{lem}
\label{lem:wt+cr}
The map $\Phi$ preserves the sum of the weight and the number of crossings, i.e.,
$$\wt(\mm)+\cross(\mm)=\wt(\Phi(\mm))+\cross(\Phi(\mm)).$$
\end{lem}
\begin{proof}
  If m is fixed by $\Phi$, then the assertion is clear. Assume that an edge $e$ of $\mm$ is toggled by $\Phi$.  We may also assume that $e$ is marked. The block difference of an edge changes after unmarking $e$ if and only if the edge crosses $e$. Let $e'$ (respectively, $e''$) be an edge that crosses $e$ from the left (respectively, right). Then unmarking $e$ decreases $\bdiff_\mm(e')$ by 1 and increases $\bdiff_\mm(e'')$ by 1. Since $e$ is convertible, we must have $\bdiff_\mm(e')>0$ and $\bdiff_\mm(e'')<-1$, and $\wt(\Phi(\mm))$ is equal to $\wt(\mm)$ subtracted by the number of edges that
  cross $e$.

  On the other hand, by definition of  $\cross(\mm)=\cross(\pi\setminus S)-\cross(\pi|_S)$, we actually count a crossing of pair of unmarked edges as $+1$, a crossing of pair of marked edges as $-1$ and a crossing of pair of a unmarked edge and a marked edge as $0$. Thus, when we toggle $e$ from marked to unmarked, for every marked (respectively, unmarked) edge that crosses $e$ its contribution to $\cross(\mm)$ changes from $-1$ to $0$
  (respectively, from $0$ to $1$). Therefore $\cross(\Phi(\mm))$ is equal to the
  sum of $\cross(\mm)$ and the number of edges that cross $e$. This together
  with the result in the above paragraph implies the conclusion.
\end{proof}

Finally we give a proof of Theorem~\ref{thm:Main2}.

\begin{proof}[Proof of Theorem~\ref{thm:Main2}]
  Recall from \eqref{eq:cancnotfree} that we have
  \[
C(n_1,\dots,n_k) =\sum_{\mm\in \PM^{*}_{n_1,\dots,n_k}}(-1)^{\ee(\mm)}y^{\bwex(\mm)}q^{\wt(\mm)+\cross(\mm)}.
  \]
  Lemmas~\ref{lem:involution}, \ref{lem:bwex} and \ref{lem:wt+cr} imply
  that $\Phi$ is a sign-reversing and weight-preserving involution on
  $\PM^{*}_{n_1,\dots,n_k}$ with fixed point set $\DD(n_1,\dots,n_k)$.
Thus
  \[
    C(n_1,\dots,n_k) =\sum_{\mm\in
      \DD(n_1,\dots,n_k)}(-1)^{\ee(\mm)}y^{\bwex(\mm)}q^{\wt(\mm)+\cross(\mm)}.
  \]
  If $\mm=(\pi,S)\in \DD(n_1,\dots,n_k)$, then $S=E(\pi)$ and therefore
  $\ee(\mm)=0$. Thus we obtain the desired formula. 
\end{proof}

\section{Further study}
\label{sec:further-study}

Pan and Zeng \cite{PanZeng} introduced the $q$-Laguerre polynomials
$L^{(\alpha)}_n(x;y,q)$ with an additional parameter $\alpha$ defined by
$L^{(\alpha)}_{-1}(x;y,q)=0$, $L^{(\alpha)}_0(x;y,q)=1$, and for $n\ge1$,
\[
  L^{(\alpha)}_{n+1}(x;y,q) = (x-(y([n+\alpha+1]_q+[n]_q)))L^{\alpha}_n(x;y,q)
  -y[n]_q[n+\alpha]_q L^{(\alpha)}_{n-1}(x;y,q).
\]
They found combinatorial interpretations for the polynomials
$L^{(\alpha)}_n(x;y,q)$ and their moments. They also showed that if $\alpha$ is
a nonnegative integer then the linearization coefficient
\begin{equation}
  \label{eq:alpha}
\LL_{\alpha,y,q}(L^{(\alpha)}_{n_1}(x;y,q)\dots L^{(\alpha)}_{n_k}(x;y,q))  
\end{equation}
is a polynomial in $y$ and $q$ with nonnegative integer coefficients, where
$\LL_{\alpha,y,q}$ is the linear functional for the orthogonal polynomials
$L^{(\alpha)}_n(x;y,q)$.

If $\alpha=0$, then $L^{(0)}_n(x;y,q)$ is the $q$-Laguerre polynomial
$L_n(x;q,y)$ that we considered in this paper. Therefore,
\begin{equation}
  \label{eq:alpha=0}
  \LL_{0,y,q}(L^{(0)}_{n_1}(x;y,q)\dots L^{(0)}_{n_k}(x;y,q))
  = \sum_{\sigma\in \mathcal{D}
  (n_1,\dots,n_k)} y^{\wex(\sigma)}q^{\CR(\sigma)}. 
\end{equation}
If $q=1$, we have
\begin{equation}
  \label{eq:q=1}
  \LL_{\alpha,y,1}(L^{(\alpha)}_{n_1}(x;y,1)\dots L^{(\alpha)}_{n_k}(x;y,1))
  = \sum_{\sigma\in \mathcal{D}
  (n_1,\dots,n_k)} y^{\wex(\sigma)} (\alpha+1)^{\mathrm{cyc}(\sigma)},
\end{equation}
where $\mathrm{cyc}(\sigma)$ is the number of cycles of $\sigma$. 
It is an open problem to find a combinatorial interpretation for
\eqref{eq:alpha} generalizing both \eqref{eq:alpha=0} and \eqref{eq:q=1}.

In this paper we found a sign-reversing involution proving \eqref{eq:alpha=0}.
Unfortunately, our approach does not apply directly to \eqref{eq:q=1}. If
$y=q=1$, then a simple combinatorial proof of \eqref{eq:q=1} was found by Foata
and Zeilberger \cite{FZ84}. This was also proved by Kim and Zeng \cite{KZ} who
found a combinatorial interpretation for linearization coefficients of general
Sheffer polynomials using sign-reversing involutions. For the case $q=1$,
M\'edicis \cite{Med98} gave a combinatorial proof of a generalization of
\eqref{eq:q=1} to Meixner polynomials using sign-reversing involutions.

Finally, we note that there is a combinatorial interpretation of the
linearization coefficient $\LL(P_{n_1}(x)\dots P_{n_k}(x))$ for any orthogonal
polynomials in terms of weighted Motzkin paths due to de M\'edicis and Stanton
\cite{MedicisStanton}. Their result immediately implies the nonnegativity of
\eqref{eq:alpha}. It might be interesting to study the linearization
coefficients of the $q$-Laguerre polynomials considered here using their
combinatorial model.

\section*{Acknowledgments}

The authors would like to thank the anonymous referee for providing relevant references and useful comments.


\end{document}